\documentclass[preprint,11pt]{elsarticle}

\usepackage[a4paper]{geometry}
\usepackage{amsmath, amsthm, amssymb, mathtools, tikz, enumitem, booktabs}
\usepackage[colorlinks,allcolors=black]{hyperref}
\usepackage[bf,singlelinecheck=false]{caption}
\usepackage[ruled]{algorithm}
\usepackage[noend]{algpseudocode}
\usepackage{algorithmicx, comment}
\usepackage{aliascnt}
\usepackage{multirow}

\theoremstyle{definition}
\newtheorem{theorem}{Theorem}[section]
\newtheorem{definition}[theorem]{Definition}
\newtheorem{lemma}[theorem]{Lemma}
\newtheorem{cor}[theorem]{Corollary}
\newtheorem{remark}[theorem]{Remark}
\newtheorem{example}[theorem]{Example}

\makeatletter
\let\c@algorithm\relax
\let\c@figure\relax
\let\c@table\relax
\let\c@theorem\relax
\makeatother
\newaliascnt{algorithm}{common}
\newaliascnt{figure}{common}
\newaliascnt{table}{common}
\newaliascnt{theorem}{common}

\usetikzlibrary{arrows, shapes, patterns, decorations.pathmorphing}
\newcommand{\arc}[2]{\path{} (#1) edge [->,thick] node {} (#2);}
\newcommand{\arcWavy}[2]{\path{} (#1) edge [->,style={decorate, decoration=snake}] node {} (#2);}
\newcommand{\arcThick}[2]{\path{} (#1) edge [->,line width=0.75mm] node {} (#2);}
\newcommand{\arcSym}[2]{\path{} (#1) edge [<->,thick] node {} (#2);}
\newcommand{\arcDash}[2]{\path{} (#1) edge [->,dashed,thick] node {} (#2);}
\newcommand{\arcSymDash}[2]{\path{} (#1) edge [<->,dashed,thick] node {} (#2);}
\newcommand{\arcSymDot}[2]{\path{} (#1) edge [<->,dotted,thick] node {} (#2);}
\newcommand{\looparcL}[1]{\path{} (#1) edge [loop left,thick] node {} (#1);}

\MakeRobust{\Call}

\def\<{\langle}
\def\>{\rangle}
\newcommand{\idOmega}{\textrm{id}_{\Omega}}
\newcommand{\set}[2]{\{#1\,:\,#2\}}
\newcommand{\N}{\mathbb{N}}
\newcommand{\n}{\{1,\ldots,n\}}

\newcommand{\labelSet}{\mathfrak{L}}
\newcommand{\labelFunc}{\textsc{Label}}
\newcommand{\exLabel}[1]{\textit{#1}}
\newcommand{\stackS}{S}
\newcommand{\stackT}{T}
\newcommand{\stackU}{U}
\newcommand{\stackV}{V}
\def\EVLA{\textsf{EVLA}}
\newcommand{\nauty}{\textsc{nauty}}
\newcommand{\bliss}{\textsc{bliss}}
\newcommand{\approxFunc}{\textsc{Approx}}
\newcommand{\approxCanonical}{\textsc{Approx}_{\textsc{C}}}
\newcommand{\approxStrong}{\textsc{Approx}_{\textsc{S}}}
\newcommand{\approxWeak}{\textsc{Approx}_{\textsc{W}}}
\newcommand{\splitFunc}{\textsc{Split}}
\newcommand{\fixedFunc}{\Call{Fixed}{}}
\newcommand{\fixedCanonical}{\Call{Fixed}{}_{\textsc{C}}}
\newcommand{\fixedStrong}{\Call{Fixed}{}_{\textsc{S}}}
\newcommand{\fixedWeak}{\Call{Fixed}{}_{\textsc{W}}}
\newcommand{\canonFunc}{\textsc{Canon}}
\newcommand{\equitableFunc}{\textsc{Equitable}}
\newcommand{\Sn}[1]{\mathcal{S}_{#1}}

\newcommand\Approx[2]{\ifthenelse{\equal{#2}{}}
                         {\textsc{Approx} (#1)}
                         {\Call{Approx}{#1,#2}}}
\newcommand{\Split}[2]{\ifthenelse{\equal{#2}{}}
                        {\Call{Split}{#1}}
                        {\Call{Split}{#1,#2}}}
\newcommand{\Fixed}[1]{\Call{Fixed}{#1}}
\newcommand{\Canon}[1]{\Call{Canon}{#1}}
\newcommand{\Equitable}[1]{\Call{Equitable}{#1}}

\newcommand{\Sym}[1]{\operatorname{Sym}\!\left(#1\right)}
\newcommand{\Auto}[1]{\operatorname{Aut}\!\left(#1\right)}
\newcommand{\Iso}[2]{\ifthenelse{\equal{#2}{}}
                      {\operatorname{Iso}\!\left(#1\right)}
                      {\operatorname{Iso}\!\left(#1,#2\right)}}

\newcommand{\Stacks}[1]{\operatorname{\textsc{DigraphStacks}}\!\left(#1\right)}
\newcommand{\Squash}[1]{\Call{Squash}{#1}}
\newcommand{\LabelledDigraphs}[2]{\operatorname{\textsc{LabelledDigraphs}}
                                  \!\left(#1,#2\right)}
\newcommand{\EmptyStack}[1]{\operatorname{\textsc{EmptyStack}}\!\left(#1\right)}

\algnewcommand\algorithmicswitch{\textbf{switch}}
\algnewcommand\algorithmiccase{\textbf{case}}
\algdef{SE}[SWITCH]{Switch}{EndSwitch}[1]{\algorithmicswitch\ #1
\ \algorithmicdo{}
}{\algorithmicend\ \algorithmicswitch}
\algdef{SE}[CASE]{Case}{EndCase}[1]{\algorithmiccase\ #1:}{\algorithmicend\
\algorithmiccase}
\algtext*{EndSwitch}
\algtext*{EndCase}

\begin{document}

\begin{frontmatter}

\title{Permutation group algorithms based on directed graphs}

\author{Christopher Jefferson}
\ead{caj21@st-andrews.ac.uk}
\ead[url]{https://caj.host.cs.st-andrews.ac.uk}

\author{Markus Pfeiffer}
\ead{markus.pfeiffer@st-andrews.ac.uk}
\ead[url]{https://www.morphism.de/~markusp}

\author{Wilf A.\ Wilson}
\address{University of St~Andrews\\School of Computer Science\\North Haugh\\St Andrews\\KY16 9SX\\Scotland}
\ead{waw7@st-andrews.ac.uk}
\ead[url]{https://wilf.me}

\author{Rebecca Waldecker}
\address{Martin-Luther-Universit\"at Halle-Wittenberg\\Institut f\"ur Mathematik\\06099 Halle\\Germany}
\ead{rebecca.waldecker@mathematik.uni-halle.de}
\ead[url]{https://www2.mathematik.uni-halle.de/waldecker/index-english.html}

\journal{Journal of Algebra}

\begin{abstract}
  We introduce a new framework for solving an important class of computational
  problems involving finite permutation groups, which includes calculating set
  stabilisers, intersections of subgroups, and isomorphisms of combinatorial
  structures.
  Our techniques are inspired by and generalise `partition backtrack', which is the
  current state-of-the-art algorithm introduced by Jeffrey Leon in 1991.
  But, instead of ordered partitions, we use labelled
  directed graphs to organise our backtrack search algorithms,
  which allows for a richer representation of many problems
  while often resulting in smaller search spaces.
  In this article we present the theory underpinning our framework, we describe
  our algorithms, and we show the results of some experiments.
  An implementation of our algorithms is available as free software in the
  \textsc{GraphBacktracking} package for \textsc{GAP}.
\end{abstract}

\end{frontmatter}

\section{Introduction}\label{sec-intro}

Many of the most important problems in computational permutation group theory
can be phrased as search problems, where we typically search for the
intersection of subsets of a symmetric group.
Standard problems that match this description are
the computation of point stabilisers or set stabilisers, of transporter sets, or
of normalisers or centralisers of subgroups.
Searching for automorphisms and isomorphisms of a wide range of combinatorial structures can be done in this way as well as deciding whether or not
combinatorial objects are in the same orbit under some group action, as is the case
for
element and subgroup conjugacy.

For many of these problems, the best known way to solve them is based on
Leon's \emph{partition backtrack algorithm} (see~\cite{leon1991}), which often performs excellently, but has exponential worst-case complexity.
Leon's algorithm conducts a backtrack search through the
elements of the symmetric group, which it organises around a collection of
ordered partitions (see~\cite{directorscut} for more details).
By encoding information about the given problem into those partitions, it is possible to cleverly prune (i.e., omit superfluous parts of) the
search space.
There have already been some extensions and improvements, inspired by graph-based ideas of {M}c{K}ay
(see for example~\cite{theissen} and~\cite{newrefiners}). This leads us to believe that
more powerful pruning, and ultimately better
performance, could be obtained by using graphs directly,
at the expense of the increased computation required at
each node of the remaining search.
In the present paper, we therefore place labelled directed graphs at the heart of backtrack search algorithms. 

The basic idea is parallel to that described above for Leon's algorithm:
When we search for an intersection of subsets of the symmetric group, we find suitable labelled digraph stacks such that the intersection
can be viewed as the set of isomorphisms (induced from the symmetric group) from the first labelled digraph stack to the second (see Section~\ref{sec-stacks}).
We encode information about the subsets into the labelled digraph stacks with
refiners (see Section~\ref{sec-refiners}),
and we just remark here that this generalises partition backtrack,
because partition backtrack can be viewed as using vertex-labelled digraphs without arcs,
where the vertex labels are in one-to-one correspondence with the cells of the ordered
partitions.
Approximators capture the fact that we typically overestimate the set of isomorphisms rather than calculate it exactly (see Section~\ref{sec-approximations}).
The last ingredient comes into play when our approximation indicates that the refiners have encoded as much information as possible in a given moment and cannot restrict the search space further.
Then we divide the search into smaller areas by defining new labelled digraph stacks, which is known as splitting (see Section~\ref{sec-splitter}). 
We discuss algorithms based on the method just described and prove their correctness in
Sections~\ref{sec-search} and~\ref{sec-r-base},
and in Section~\ref{sec-experiments}  we give details
of various experiments that compare our algorithms with the
current state-of-the-art techniques.
We conclude, in Section~\ref{sec-end}, with brief comments on the results of
this paper and the directions that they suggest for further investigation.

Finally, we would like to mention that we expect the reader to be familiar with basic concepts of
permutation group theory and graph theory and that we only briefly explain our notation before moving to the main content. 
There is an extended version of this article~\cite{directorscut},
where we give proofs that are omitted here, along with
much more detail and background information. We also include additional examples there and new material that is currently in preparation for a separate publication.

\subsection*{Acknowledgements}

The authors would like to thank the DFG (\textbf{Grant no.~WA 3089/6-1}) and the
Volks\-wagenstiftung (\textbf{Grant no.~93764}) for financially supporting this work and projects leading up to it.
The first and third authors are supported by the Royal Society
(\textbf{Grant codes RGF\textbackslash EA\textbackslash 181005
                 and URF\textbackslash R\textbackslash 180015}).
Special thanks go to Paula H\"ahndel and Ruth Hoffmann for frequent
discussions on topics related to this work, and for suggestions on how to improve this paper.
Finally, the authors thank the referees for reading our drafts carefully and for making many helpful suggestions.

\section{Preliminaries}\label{sec-defs-prelims}

Throughout this paper, $\Omega$ denotes some finite totally-ordered set on which
we define all of our groups, digraphs, and related objects.  For example, every
group in this paper is a finite permutation group on $\Omega$, i.e., a subgroup
of $\Sym{\Omega}$, the symmetric group on $\Omega$.
We follow the standard
group-theoretic notation and terminology from the literature, such as that used
in~\cite{dixonmortimer}, and write $\cdot$ for the composition of maps in
$\Sym{\Omega}$, or we omit a symbol for this binary operation altogether. We
write $\N$ for the set $\{1, 2, 3, \ldots\}$ of all natural numbers, and $\N_{0}
\coloneqq \N \cup \{0\}$.
If $n \in \N$, then $\Sn{n} \coloneqq \Sym{\{1,\ldots,n\}}$.
For many types of objects that we define on $\Omega$, for example lists, sets, or graphs, we give a way of applying
elements of $\Sym{\Omega}$ to them (denoted by exponentiation) in a structure-preserving way.

Let $\Gamma$ and $\Delta$ be digraphs (which is short for directed graphs) 
with vertex set $\Omega$.
Then we say that a permutation $g \in \Sym{\Omega}$ \emph{induces an
isomorphism from $\Gamma$ to $\Delta$} if and only if it defines a
structure-preserving map from $\Gamma$ to $\Delta$, in which case we write
$\Gamma^{g} = \Delta$.

We use the notation $\Iso{\Gamma}{\Delta}$ for the set of
isomorphisms from $\Gamma$ to $\Delta$ that are induced by elements of
$\Sym{\Omega}$.  If $\Iso{\Gamma}{\Delta}$ is non-empty, then we
call $\Gamma$ and $\Delta$ \emph{isomorphic}.
Similarly, we write $\Auto{\Gamma}
\coloneqq \Iso{\Gamma}{\Gamma}$
for the subgroup of $\Sym{\Omega}$ consisting of all
elements that induce \emph{automorphisms} of $\Gamma$.

\subsection{Labelled digraphs}

Our techniques for searching in $\Sym{\Omega}$ are built around digraphs in
which each vertex and arc (i.e.~directed edge) is given a \textit{label} from a set of labels
$\labelSet$. We
define a \emph{vertex- and arc-labelled digraph}, or \emph{labelled digraph} for
short, to be a triple $(\Omega, A, \labelFunc)$, where $(\Omega, A)$ is a
digraph and $\labelFunc$ is a function from $\Omega \cup A$ to $\labelSet$.
More precisely, for any $\delta \in \Omega$ and $(\alpha, \beta) \in A$, the
label of the vertex $\delta$ is $\labelFunc(\delta) \in \labelSet$, and the
label of the arc $(\alpha, \beta)$ is $\labelFunc(\alpha, \beta) \in
  \labelSet$.  We call such a function a \emph{labelling function}.  We point out
that our notion of digraphs is very general and that loops are allowed.
(For more details and some examples see~\cite{directorscut}.) 

We fix
$\labelSet$ as some non-empty set that contains every label that we require and
serves as the codomain of every labelling function.  
For the equitable vertex labelling algorithm discussed in Section~\ref{sec-equitable}, we require
some arbitrary but fixed total ordering on $\labelSet$.

The symmetric group on $\Omega$ acts on the sets of graphs and digraphs with
vertex
set $\Omega$, respectively, and on their labelled variants, in a natural way.
We give more details about this for labelled digraphs; the forthcoming notions
are defined analogously for the other kinds of graphs and digraphs that we have
mentioned.  Let $\LabelledDigraphs{\Omega}{\labelSet}$ denote the class of
labelled digraphs on $\Omega$ with labels in $\labelSet$, and let
$\Gamma = (\Omega, A, \labelFunc) \in \LabelledDigraphs{\Omega}{\labelSet}$
and $g \in \Sym{\Omega}$.  Then we define
$\Gamma^{g} = (\Omega, A^{g}, \labelFunc^{g}) \in
\LabelledDigraphs{\Omega}{\labelSet}$, where:
\begin{enumerate}[label=\textrm{(\roman*)}]
  \item
    $A^{g} \coloneqq \set{(\alpha^{g}, \beta^{g})}{(\alpha, \beta) \in A}$,

  \item
    $\labelFunc^{g}(\delta) \coloneqq
          \labelFunc(\delta^{g^{-1}})$ for all $\delta \in \Omega$, and

  \item
    $\labelFunc^{g}(\alpha, \beta) \coloneqq
     \labelFunc(\alpha^{g^{-1}}, \beta^{g^{-1}})$
    for all $(\alpha, \beta) \in A^{g}$.
\end{enumerate}
In other words, the arcs are mapped according to $g$, and the label of a vertex
or arc in $\Gamma^{g}$ is the label of its preimage in $\Gamma$.
This gives rise to a group action of $\Sym{\Omega}$ on
$\LabelledDigraphs{\Omega}{\labelSet}$.

\section{Stacks of labelled digraphs}\label{sec-stacks}

In this section we introduce labelled digraph stacks, with the rough idea in mind that
we use them to approximate the set of permutations we search for.
More precisely, 
we attempt to choose suitable labelled digraph stacks in such a way that the set of isomorphisms
from one to the other approximates the set we search for as closely as possible.

A \emph{labelled digraph stack} on $\Omega$ is a finite (possibly empty) list of
labelled digraphs on $\Omega$.  We denote the collection of all labelled digraph
stacks on $\Omega$ by $\Stacks{\Omega}$.  The \emph{length} of a labelled
digraph stack $\stackS$, written as $|\stackS|$, is the number of entries that it
contains.  A labelled digraph stack of length $0$ is called \emph{empty}, and
is denoted by $\EmptyStack{\Omega}$.
We use notation typical for lists, whereby if $i \in \{1, \ldots,
|\stackS|\}$, then $\stackS[i]$ denotes the $i^{\text{th}}$ labelled digraph in
the stack $\stackS$.

We allow any labelled digraph stack on $\Omega$ to be appended
onto the end of another.  If $\stackS, \stackT \in
\Stacks{\Omega}$ have lengths $k$ and $l$, respectively, then we define $\stackS
\Vert \stackT$ to be the labelled digraph stack
$\left[
    \stackS[1], \ldots, \stackS[k],
    \stackT[1], \ldots, \stackT[l]
    \right]$
of length $k + l$.

We define an action of $\Sym{\Omega}$ on $\Stacks{\Omega}$  via the action of
$\Sym{\Omega}$ on the set of all labelled digraphs on $\Omega$.  More
specifically, for all $\stackS \in \Stacks{\Omega}$ and $g \in \Sym{\Omega}$, we
define $\stackS^{g}$ to be the labelled digraph stack of length $|\stackS|$ with
$\stackS^{g}[i] = {\stackS[i]}^{g}$ for all $i \in \{1, \ldots, |\stackS|\}$.
An \emph{isomorphism} from $\stackS$
to another labelled digraph stack $\stackT$ (induced by $\Sym{\Omega}$) is
therefore a permutation $g \in \Sym{\Omega}$ such that $\stackS^{g} = \stackT$.
In particular, only digraph stacks of equal lengths can be isomorphic.
We note that every permutation in
$\Sym{\Omega}$ induces an automorphism of $\EmptyStack{\Omega}$.
As we do with digraphs, we use the notation
$\Iso{\stackS}{\stackT}$  for the set of isomorphisms
from the stack $\stackS$ to the stack $\stackT$ induced by elements of
$\Sym{\Omega}$, and $\Auto{\stackS}$ for the group of automorphisms
of $\stackS$ induced by elements of $\Sym{\Omega}$.

\begin{remark}\label{rmk-stack-iso-auto}
Let $\stackS, \stackT, \stackU, \stackV \in
  \Stacks{\Omega}$. It follows from the definitions that
\[
  \Iso{\stackS}{\stackT} =
  \begin{cases}
    \varnothing
     & \text{if\ } |\stackS| \neq |\stackT|, \\
    \bigcap_{i = 1}^{|\stackS|} \Iso{\stackS[i]}{\stackT[i]}
     & \text{if\ } |\stackS| = |\stackT|,
  \end{cases}
  \quad\text{and that\ }
  \Auto{\stackS} = \bigcap_{i = 1}^{|\stackS|}
  \Auto{\stackS[i]}.
\]
In addition $\Auto{\stackS \mathop{\Vert} \stackU} \leq \Auto{\stackS}$,
and if $|\stackS| = |\stackT|$, then
$\Iso{\stackS \mathop{\Vert} \stackU}{\stackT \mathop{\Vert} \stackV}
  \subseteq
  \Iso{\stackS}{\stackT}$.
Roughly speaking, the automorphism group of a labelled digraph stack, and the
set of isomorphisms from one labelled digraph stack to another one of equal
length, become potentially smaller as new entries are added to the stacks.
\end{remark}

We illustrate some of the foregoing concepts in Example~\ref{ex-stack}.
Since we use an orbital graph, we briefly recall (see for example~\cite[Section~3.2]{dixonmortimer}):

\begin{definition}[Orbital graph]
  Let $G \leq \Sym{\Omega}$, and let $\alpha, \beta \in \Omega$ be such that
  $\alpha \neq \beta$.  Then the \emph{orbital graph of $G$ with base-pair
    $(\alpha, \beta)$} is the digraph
  $\left(\Omega,\,\set{(\alpha^{g}, \beta^{g})}{g \in G}\right)$.
\end{definition}

\begin{example}\label{ex-stack}
  Let $\Omega = \{1, \ldots, 6\}$.  Here we define a
  labelled digraph stack $\stackS$ on $\Omega$ that has length $3$, by
  describing each of its members.

  We define the first entry of $\stackS$ via the orbital graph of $K \coloneqq
  \<(1\,2)(3\,4)(5\,6), (2\,4\,6) \>$ with base-pair $(1, 3)$.
  The automorphism
  group of this orbital graph (as always, induced by $\Sym{\Omega}$) is $K$
  itself; in other words, this orbital graph perfectly represents $K$ via its
  automorphism group.
  In order to define $\stackS[1]$, we convert this orbital graph into a labelled
  digraph by assigning the label \exLabel{white} to each vertex and
  assigning the label \exLabel{solid} to each arc.
  This does not change the automorphism group of the digraph.

  We define the second entry of $\stackS$ to be the labelled digraph on $\Omega$
  without arcs, whose vertices $1$ and $2$ are labelled \exLabel{black}, and
  whose remaining vertices are labelled \exLabel{white}.  The automorphism group
  of this labelled digraph is the setwise stabiliser of $\{1, 2\}$ in
  $\Sym{\Omega}$.

  We define the third entry of $\stackS$ to be the labelled digraph $\stackS[3]$
  shown in Figure~\ref{fig-ex-stack}, with arcs and labels chosen from the
  set
  $\{\exLabel{black},\, \exLabel{white},\, \exLabel{solid},\, \exLabel{dashed}\,\}$
  as depicted there;
  its automorphism group is $\< (1\,2), (3\,4)(5\,6) \>$.

  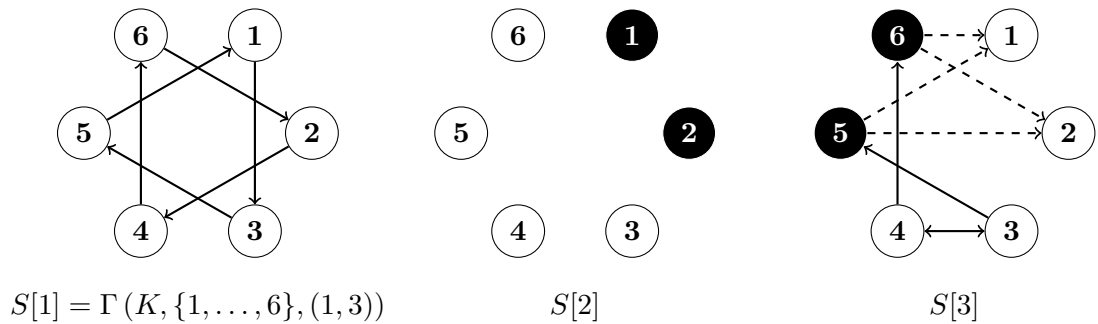
\begin{figure}[!ht]
    \centering
    \begin{tikzpicture}
      \foreach \x in {1,2,3,4,5,6} {
          \tikzstyle{white}=[circle, draw=black]
          \node[white] (\x) at (-\x*60+120:1.5cm) {$\mathbf{\x}$};};
      \foreach \x/\y in {1/3,2/4,3/5,4/6,5/1,6/2} {\arc{\x}{\y}};
      \node at (0, -2.3)
      {$\stackS[1] = \Gamma\left(K, \{1, \ldots, 6\}, (1, 3)\right)$};
    \end{tikzpicture}
    \quad
    \begin{tikzpicture}
      \tikzstyle{white}=[circle, draw=black]
      \tikzstyle{black}=[circle, draw=white, fill=black!100]
      \foreach \x in {1,2} {
          \node[black] (\x) at (-\x*60+120:1.5cm) {$\color{white}\mathbf{\x}$};};
      \foreach \x in {3,4,5,6} {
          \node[white] (\x) at (-\x*60+120:1.5cm) {$\mathbf{\x}$};};
      \node at (0, -2.3) {$\stackS[2]$};
    \end{tikzpicture}
    \qquad\quad
    \begin{tikzpicture}
      \tikzstyle{white}=[circle, draw=black]
      \tikzstyle{black}=[circle, draw=white, fill=black!100]
      \foreach \x in {1,2,3,4} {
          \node[white] (\x) at (-\x*60+120:1.5cm) {$\mathbf{\x}$};};
      \foreach \x in {5,6} {
          \node[black] (\x) at (-\x*60+120:1.5cm) {$\color{white}\mathbf{\x}$};};
      \arcDash{5}{1}
      \arcDash{6}{1}
      \arcDash{5}{2}
      \arcDash{6}{2}
      \arc{3}{5}
      \arc{4}{6}
      \arcSym{3}{4}
      \node at (0, -2.3) {$\stackS[3]$};
    \end{tikzpicture}
    \caption{\label{fig-ex-stack}Diagrams of the labelled digraphs in the
      labelled digraph stack $\stackS$ from Example~\ref{ex-stack}. The
      vertices and arcs of these labelled digraphs are styled according to their
      labels, which are chosen from the set $\{\exLabel{black},\,\exLabel{white},\,
        \exLabel{solid},\, \exLabel{dashed}\,\}$.}
  \end{figure}
  Given the automorphism groups of the individual entries of $\stackS$, as
  described above, it follows that the automorphism group of $\stackS$  consists
  of precisely those elements of $K$ that stabilise the set $\{1, 2\}$, and that
  are automorphisms of the labelled digraph $\stackS[3]$. Hence this group is
  $\< (1\,2)(3\,4)(5\,6) \>$.
  Since $(1\,2)$ is an automorphism of $\stackS[2]$ and $\stackS[3]$, but not of
  $\stackS[1]$, it follows that $\stackS^{(1\,2)} = [{\stackS[1]}^{(1\,2)},
  \stackS[2], \stackS[3]] \neq \stackS$.
  We also note that $\Iso{\stackS}{\stackS^{(1\,2)}}$ is the right coset
  $\Auto{\stackS} \cdot (1\,2) = \{ (1\,2), (3\,4)(5\,6) \}$ of $\Auto{\stackS}$
  in $\Sym{\Omega}$.
\end{example}

\subsection{The squashed labelled digraph of a stack}\label{sec-squashed-stack}

For our exposition in Section~\ref{sec-approximations},
it is convenient to have a labelled digraph whose automorphism group is equal
to that of a given labelled digraph stack. This is analogous to the final entry
of an ordered partition stack~\cite[Section~4]{leon1991}.
This special labelled digraph is a new object defined from
the stack, but it is not part of the stack itself.

For this we fix a symbol $\#$
that is never to be used as the label of a vertex or an arc
in any labelled digraph.

\begin{definition}\label{defn-squashed}
  Let $\stackS$ be a labelled digraph stack on $\Omega$, with $\stackS[i]
  \coloneqq (\Omega, A_{i}, \labelFunc_{i})$ being some labelled digraph on
  $\Omega$ for each $i \in \{1, \ldots, |\stackS|\}$.  Then the \emph{squashed
  labelled digraph} of $\stackS$, denoted by $\Squash{\stackS}$, is the labelled
  digraph $(\Omega, A, \labelFunc)$, where
  \begin{itemize}
    \item

      $A = \bigcup_{i = 1}^{|\stackS|} A_{i}$,

    \item

      $\labelFunc(\delta) =
        [\labelFunc_{1}(\delta), \ldots,
        \labelFunc_{|\stackS|}(\delta)]$ for all $\delta \in \Omega$,
      and

    \item

        $\labelFunc(\alpha, \beta)$ is the list of length $|\stackS|$
        for all $(\alpha, \beta) \in \bigcup_{i = 1}^{|\stackS|} A_{i}$,
        where
        \[\labelFunc(\alpha, \beta)[i] =
          \begin{cases}
            \labelFunc_{i}(\alpha, \beta)
             & \text{\ if\ } (\alpha, \beta) \in A_{i}, \\
            \#
             & \text{\ if\ } (\alpha, \beta) \not\in A_{i},
          \end{cases}
          \quad\text{for all}\ i \in \{1, \ldots, |\stackS|\}.
        \]
  \end{itemize}
\end{definition}

Note that the labelling function of a squashed labelled digraph of a stack can
be used to
reconstruct all information about the stack from which it was created.  We
also point out that $\Squash{\stackS}^{g} = \Squash{\stackS^{g}}$ for all
$\stackS \in \Stacks{\Omega}$ and $g \in \Sym{\Omega}$.
Therefore the following lemma holds.

\begin{lemma}\label{lem-squash-same-iso}
  Let $\stackS, \stackT \in \Stacks{\Omega}$.  Then
  \[\Iso{\stackS}{\stackT}
  =
  \Iso{\Squash{\stackS}}{\Squash{\stackT}}.\]
\end{lemma}

\begin{example}\label{ex-squashed}
  Let $\stackS$ be the labelled digraph stack from
  Example~\ref{ex-stack}. Since $|\stackS| = 3$, the labels
  in $\Squash{\stackS}$ are lists of length $3$. The vertex labels in
  $\Squash{\stackS}$ are:
  \begin{itemize}
    \item
          $\labelFunc(1) = \labelFunc(2) =
            [\exLabel{white}, \exLabel{black}, \exLabel{white}]$,
          shown as \exLabel{black} in Figure~\ref{fig-ex-squashed},
    \item
          $\labelFunc(3) = \labelFunc(4) =
            [\exLabel{white}, \exLabel{white}, \exLabel{white}]$,
          shown as \exLabel{white} in Figure~\ref{fig-ex-squashed},
          and
    \item
          $\labelFunc(5) = \labelFunc(6) =
            [\exLabel{white}, \exLabel{white}, \exLabel{black}]$,
          shown as \exLabel{grey} in Figure~\ref{fig-ex-squashed}.
  \end{itemize}
  There are ten arcs in $\Squash{\stackS}$, which in total have five
  different labels:
  \begin{itemize}
    \item
          $\labelFunc(1, 3) = \labelFunc(2, 4) =
            [\exLabel{solid}, \#, \#]$,
          shown as \exLabel{thin} in Figure~\ref{fig-ex-squashed},
    \item
          $\labelFunc(3, 4) = \labelFunc(4, 3) = [\#, \#,
            \exLabel{solid}]$,
          shown as \exLabel{dotted} in Figure~\ref{fig-ex-squashed},
    \item
          $\labelFunc(5, 2) = \labelFunc(6, 1) = [\#, \#,
            \exLabel{dashed}]$,
          shown as \exLabel{dashed} in Figure~\ref{fig-ex-squashed},
    \item
          $\labelFunc(3, 5) = \labelFunc(4, 6) =
            [\exLabel{solid}, \#, \exLabel{solid}]$,
          shown as \exLabel{thick} in Figure~\ref{fig-ex-squashed}, and
    \item
          $\labelFunc(5, 1) = \labelFunc(6, 2) =
            [\exLabel{solid}, \#, \exLabel{dashed}]$,
          shown as \exLabel{wavy} in Figure~\ref{fig-ex-squashed}.
  \end{itemize}

  Since automorphisms of labelled digraphs preserve the sets of vertices with
  any particular label, it is clear that $\Auto{\Squash{\stackS}} \leq \<
  (1\,2), (3\,4), (5\,6) \>$. This containment is proper, since
  $\Auto{\Squash{\stackS}} = \Auto{\stackS}$ by Lemma~\ref{lem-squash-same-iso},
  and $\Auto{\stackS} = \<(1\,2)(3\,4)(5\,6)\>$, as discussed in
  Example~\ref{ex-stack}.
  Indeed, inspection of the arc labels in $\Squash{\stackS}$ shows that any
  automorphism that interchanges the pair of points in any of $\{1, 2\}$, $\{3,
  4\}$, or $\{5, 6\}$ also interchanges the other pairs.
  \begin{figure}[!ht]
    \centering
    \begin{tikzpicture}
      \tikzstyle{white}=[circle, draw=black]
      \tikzstyle{black}=[circle, draw=white, fill=black!100]
      \tikzstyle{grey}=[circle, draw=black, fill=gray!25]
      \foreach \x in {1,2} {
          \node[black] (\x) at (-\x*60+120:1.6cm) {$\color{white}\mathbf{\x}$};};
      \foreach \x in {3,4} {
          \node[white] (\x) at (-\x*60+120:1.6cm) {$\mathbf{\x}$};};
      \foreach \x in {5,6} {
          \node[grey] (\x) at (-\x*60+120:1.6cm) {$\mathbf{\x}$};};
      \arc{1}{3}
      \arc{2}{4}
      \arcSymDot{3}{4}
      \arcDash{5}{2}
      \arcDash{6}{1}
      \arcThick{3}{5}
      \arcThick{4}{6}
      \arcWavy{5}{1}
      \arcWavy{6}{2}
    \end{tikzpicture}
    \caption{\label{fig-ex-squashed}
      A depiction of the squashed labelled digraph
      $\Squash{\stackS}$ from Example~\ref{ex-squashed}, which is
      constructed from the labelled digraph stack $\stackS$ from
      Example~\ref{ex-stack}.}
  \end{figure}
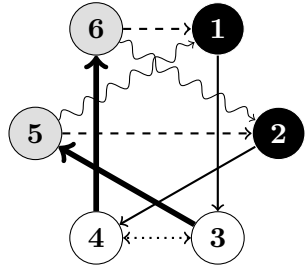
\end{example}

\section{Adding information to stacks with refiners}\label{sec-refiners}

In this section we introduce and discuss refiners for labelled digraph stacks.
We use refiners to encode information about a search problem into the stacks
around which the search is organised, in order to prune the search space.

\begin{definition}\label{defn-refiner}
  A \emph{refiner} for a set of permutations $U \subseteq \Sym{\Omega}$ is a
  pair $(f_{L}, f_{R})$ of functions from $\Stacks{\Omega}$ to itself such
  that for all isomorphic $\stackS, \stackT \in \Stacks{\Omega}$:
  \[
    U \cap \Iso{\stackS}{\stackT}
    \subseteq
    U \cap \Iso{f_{L}(\stackS)}{f_{R}(\stackT)}.
  \]
\end{definition}

Leon introduces the concept of refiners in \cite{leon1991}.
Although the word ``refiner'' might seem counter-intuitive
from the definition, Remark \ref{rmk-stack-iso-auto}
makes it clear that the stacks
$S \Vert f_{L}(\stackS)$ and $T \Vert f_{R}(\stackT)$
do indeed give rise to a closer (or ``finer'') approximation of the set we search for.

While a refiner depends on a subset of $\Sym{\Omega}$, we do not include this in
our notation in order to make it less complicated.
As a trivial example, every pair of functions from $\Stacks{\Omega}$ to itself
is a refiner for the empty set, and it is in fact relevant for practical
applications to be able to search for the empty set.  
 
In the following lemma, we formulate additional equivalent definitions of
refiners.

\begin{lemma}\label{lem-refiner-equiv-definitions}
  Let $(f_{L}, f_{R})$ be a pair of functions from $\Stacks{\Omega}$ to itself
  and let $U \subseteq \Sym{\Omega}$.  Then the following are equivalent:
  \begin{enumerate}[label=\textrm{(\roman*)}]
    \item\label{item-refiner-initial}
          $(f_{L}, f_{R})$ is a refiner for $U$.

    \item\label{item-refiner-long}
          For all isomorphic $\stackS, \stackT \in \Stacks{\Omega}$:
          \[U \cap \Iso{\stackS}{\stackT}
            =
            U \cap \Iso{\stackS \mathop{\Vert} f_{L}(\stackS)}
            {\stackT \mathop{\Vert} f_{R}(\stackT)}.\]

    \item\label{item-refiner-individual-perm}
          For $\stackS, \stackT \in \Stacks{\Omega}$ and
          \(g \in U\):
          \[\text{if } {\stackS}^{g} = \stackT,
            \text{then } {f_{L}(\stackS)}^{g} = f_{R}(\stackT).\]
  \end{enumerate}
\end{lemma}

\begin{proof}
  $\ref{item-refiner-initial}\Rightarrow\ref{item-refiner-long}$.
  Let $\stackS, \stackT \in \Stacks{\Omega}$ be isomorphic.  Then
  \(
  U \cap \Iso{\stackS}{\stackT}
  \subseteq
  U \cap \Iso{f_{L}(\stackS)}{f_{R}(\stackT)}
  \)
  by assumption, and since $\stackS$ and $\stackT$ have equal lengths,
  then
  \[
    \Iso{\stackS}{\stackT} \cap
    \Iso{f_{L}(\stackS)}{f_{R}(\stackT)}
    =
    \Iso{\stackS \mathop{\Vert} f_{L}(\stackS)}
    {\stackT \mathop{\Vert} f_{R}(\stackT)}
  \]
  by Remark~\ref{rmk-stack-iso-auto}.
  Hence
  \begin{align*}
    U \cap \Iso{\stackS}{\stackT}
     & =
    U \cap \Iso{\stackS}{\stackT}
    \cap \big(U \cap \Iso{f_{L}(\stackS)}{f_{R}(\stackT)}\big) \\
     & =
    U \cap \big(\Iso{\stackS}{\stackT}
    \cap \Iso{f_{L}(\stackS)}{f_{R}(\stackT)}\big)             \\
     & =
    U \cap \Iso{\stackS \mathop{\Vert} f_{L}(\stackS)}
    {\stackT \mathop{\Vert} f_{R}(\stackT)}.
  \end{align*}

  $\ref{item-refiner-long}\Rightarrow\ref{item-refiner-individual-perm}$.
  Let $\stackS, \stackT \in \Stacks{\Omega}$
  and let $g \in U$.  If
  $\stackS^{g} = \stackT$, then $g \in
    \Iso{\stackS}{\stackT}$ by definition, and so $g \in
    \Iso{\stackS \mathop{\Vert} f_{L}(\stackS)}{\stackT \mathop{\Vert}
      f_{R}(\stackT)}$ by assumption.
  Since $\stackS$ and $\stackT$ have equal lengths, and $\stackS
    \mathop{\Vert} f_{L}(\stackS)$ and $\stackT \mathop{\Vert} f_{R}(\stackT)$ have
  equal lengths, it follows that so too do $f_{L}(\stackS)$ and
  $f_{R}(\stackT)$.
  Then ${f_{L}(\stackS)}^{g} = f_{R}(\stackT)$,
  since for each
  $i \in \{1, \ldots, |f_{L}(\stackS)|\}$,
  \[
    {f_{L}(\stackS)[i]}^{g}
      =
      {{\left(\stackS \mathop{\Vert} f_{L}(\stackS)\right)}[|\stackS|+i]}^{g}
    =
    {\left(\stackT \mathop{\Vert} f_{R}(\stackT)\right)}[|\stackT|+i]
      =
      {f_{R}(\stackT)}[i].
  \]

  $\ref{item-refiner-individual-perm}\Rightarrow\ref{item-refiner-initial}$.
  This implication is immediate.
\end{proof}

Perhaps
Lemma~\ref{lem-refiner-equiv-definitions}\ref{item-refiner-long} most clearly
indicates the relevance of refiners to search.

Suppose we wish to search for the intersection $U_{1} \cap \cdots \cap
U_{n}$ of some subsets of $\Sym{\Omega}$.  Let $i \in \{1,\ldots,n\}$, let $(f_{L}, f_{R})$ be a refiner for
$U_{i}$, and let $\stackS$ and $\stackT$ be isomorphic labelled digraph
stacks on $\Omega$,
such that $\Iso{\stackS}{\stackT}$ overestimates (i.e.,
contains) $U_{1} \cap \cdots \cap U_{n}$.

We may use the refiner $(f_{L}, f_{R})$ to \emph{refine} the pair of stacks
$(\stackS, \stackT)$: we apply the functions $f_{L}$ and $f_{R}$, respectively,
to the stacks $\stackS$ and $\stackT$ and obtain an extended pair of stacks
$({\stackS \mathop{\Vert} f_{L}(\stackS)}, {\stackT \mathop{\Vert} f_{R}(\stackT)})$.  We call
this process \emph{refinement}.  Note that a refiner for $U_{i}$ need not
consider the other sets in the intersection.

By Lemma~\ref{lem-refiner-equiv-definitions}\ref{item-refiner-long}, the set of
induced isomorphisms $\Iso{\stackS \mathop{\Vert} f_{L}(\stackS)}{\stackT \mathop{\Vert}
f_{R}(\stackT)}$ contains the elements of $U_{i}$ that belonged to
$\Iso{\stackS}{\stackT}$. Since $U_{i}$ contains $U_{1} \cap
\cdots \cap U_{n}$, it follows that $\Iso{\stackS \mathop{\Vert} f_{L}(\stackS)}{\stackT
\Vert f_{R}(\stackT)}$ is a (possibly smaller) new overestimate for $U_{1} \cap \cdots \cap
U_{n}$ which is contained in the previous overestimate
by Remark~\ref{rmk-stack-iso-auto}.

The following straightforward example illustrates how the condition in
Lemma~\ref{lem-refiner-equiv-definitions}\ref{item-refiner-individual-perm} is
useful for showing that a pair of functions is a refiner for some set.

\begin{example}[Labelled digraph automorphism and
    isomorphism]\label{ex-perfect-digraph}

  Let $\Gamma$ and $\Delta$ be labelled digraphs on $\Omega$,
  and define constant functions $f_\Gamma$ and $f_\Delta$ on $\Stacks{\Omega}$,
  whose images are
  the length-one digraph stacks $[\Gamma]$ and $[\Delta]$, respectively. 

  Since the permutations of $\Omega$ that induce isomorphisms from $[\Gamma]$
  to $[\Delta]$
  are exactly those that induce isomorphisms from $\Gamma$ to $\Delta$, 
  it follows by
  Lemma~\ref{lem-refiner-equiv-definitions}\ref{item-refiner-individual-perm}
  that
  $(f_{\Gamma}, f_{\Delta})$ is a refiner for $\Iso{\Gamma}{\Delta}$.
  In particular,
  $(f_{\Gamma}, f_{\Gamma})$ is a refiner for $\Auto{\Gamma}$.
\end{example}

Example~\ref{ex-perfect-digraph} illustrates the principle that 
the functions of the refiner are equal if it is a refiner for a subgroup. 
The next lemma states a slightly stronger observation. 
We omit the proofs of the next few results and refer to~\cite{directorscut}.

\begin{lemma}[\mbox{cf.~\cite[Prop 2]{leon1997},~\cite[Lemma 6]{leon1991}}]\label{lem-group-refiner-symmetric}
  Let $(f_{L}, f_{R})$ be a refiner for a subset $U \subseteq \Sym{\Omega}$ that
  contains the identity map, $\idOmega$.  Then $f_{L} = f_{R}$.
\end{lemma}

Lemma~\ref{lem-refiner-equiv-definitions}\ref{item-refiner-individual-perm}
implies:

\begin{lemma}\label{lem-simple-refiner}
  Let $f$ be a function from $\Stacks{\Omega}$ to itself, and let $U$ be a
  subset of $\Sym{\Omega}$ containing $\idOmega$.
  Then $(f, f)$ is a refiner for $U$ if and only if
  \(f({\stackS}^{g}) = {f(\stackS)}^{g}\)
  for all $g \in U$ and $\stackS \in \Stacks{\Omega}$.
\end{lemma}

Next, we see that any refiner for a non-empty set can be
derived from a function that satisfies the condition in
Lemma~\ref{lem-simple-refiner}.

\begin{lemma}\label{lem-refiner-right-coset}
  Let $U$ be a non-empty subset of $\Sym{\Omega}$, fix $x \in U$,
  and let $f_L$ and $f_R$ be functions from $\Stacks{\Omega}$
  to itself.
  Then the following are equivalent:
  \begin{itemize}[leftmargin=\parindent]
    \item
      $(f_L, f_R)$ is a refiner for $U$.
    \item
      $(f_L, f_L)$ is a refiner for $U {x}^{-1}$
      and
      $f_R(\stackS) = f_L(\stackS^{x^{-1}}){}^{x}$
      for all $\stackS \in \Stacks{\Omega}$.
  \end{itemize}
  In particular,
  if $U$ is a right coset of a subgroup $G \leq \Sym{\Omega}$, then
  $(f_L, f_R)$ is a refiner for the coset $U = G x$
  if and only if
  $(f_L, f_L)$ is a refiner for the group $G$, and
  $f_R(\stackS) = f_L(\stackS^{x^{-1}}){}^{x}$ for all $\stackS \in \Stacks{\Omega}$.
\end{lemma}

For some pairs of functions, such as those in
the upcoming Example~\ref{ex-set-of-sets}, one may use the
following results to show that a pair of functions gives a refiner.

\begin{lemma}\label{lem-refiner-nice-condition}
  Let $U \subseteq \Sym{\Omega}$, and let $f_{L}$, $f_{R}$ be
  functions from $\Stacks{\Omega}$ to itself such that
  \(
    U \subseteq \Iso{f_{L}(\stackS)}{f_{R}(\stackT)}
  \)
  for all isomorphic $\stackS, \stackT \in \Stacks{\Omega}$.
  Then $(f_{L}, f_{R})$ is a refiner for $U$.
\end{lemma}

\subsection{Examples of refiners}\label{sec-examples-perfect}

Here we give several further examples of refiners for subgroups and
their cosets, for typical group theoretic problems.
We use refiners from Example~\ref{ex-perfect-list-of-sets} in our experiments of
Section~\ref{sec-grid-groups}.
The refiners given in
Examples~\ref{ex-perfect-perm-centraliser} and~\ref{ex-perfect-list-of-sets}
have in common that they perfectly capture all the information about the set
that we search for.
This is also the case for  the refiners given in Example~\ref{ex-set-of-sets}
for sets of pairwise disjoint subsets of $\Omega$,
and for sets of subsets of $\Omega$ with pairwise distinct sizes.

As we saw in Lemmas~\ref{lem-simple-refiner}
and~\ref{lem-refiner-right-coset},
the
crucial step when creating a refiner for a subgroup $G \leq \Sym{\Omega}$ or
one of its cosets is to define a function $f$ from $\Stacks{\Omega}$ to
itself such that $f(\stackS^{g}) = {f(\stackS)}^{g}$ for all $\stackS \in
\Stacks{\Omega}$ and $g \in G$.

\begin{example}[Permutation centraliser and conjugacy]\label{ex-perfect-perm-centraliser}

  For every $g \in \Sym{\Omega}$, let $\Gamma_{g}$ be the labelled digraph on
  $\Omega$ whose set of arcs is
  \(
  \set{(\alpha, \beta) \in \Omega \times \Omega}
  {\alpha^{g} = \beta},
  \)
  and in which all labels are defined to be $0$.  For every $\stackS \in
  \Stacks{\Omega}$, define
  $f_{g}(\stackS) = [\Gamma_{g}]$.
  Let $g, h \in \Sym{\Omega}$ be arbitrary.  Then $(f_{g}, f_{g})$ is a
  refiner for the centraliser of $g$ in $\Sym{\Omega}$, and $(f_{g}, f_{h})$ is a refiner
  for the set of conjugating elements $\set{x \in \Sym{\Omega}}{g^{x} = h}$.

  We illustrate one instance of this refiner.
  Let $g = (1\,2)(3\,6\,5) \in \Sn{6}$ and
  let $K$ denote the centraliser of $g$ in $\Sn{6}$.
  A diagram of $\Gamma_{g}$ is shown
  in Figure~\ref{fig-perm-centraliser}.
  Note that there is a unique loop, namely at vertex $4$,
  because $4$ is the unique fixed point of $g$ on $\{1,\ldots,6\}$.
  \begin{figure}[!ht]
    \begin{center}
      \begin{tikzpicture}
        \foreach \x in {1,2,3,4,5,6} {
          \node[circle, draw=black] (\x) at (-\x*60+120:1.5cm) {$\mathbf{\x}$};};

        \arcSym{1}{2}
        \arc{3}{6}
        \arc{6}{5}
        \arc{5}{3}
        \looparcL{4}

        \node at (3, 0) {};
        \node at (-3, 0) {$\Gamma_{g}$:};
      \end{tikzpicture}
    \end{center}

    \caption[Refiner for permutation centraliser]{
      The labelled digraph $\Gamma_{g}$ for $g = (1\,2)(3\,6\,5)$,
      from Example~\ref{ex-perfect-perm-centraliser}.
    }\label{fig-perm-centraliser}
  \end{figure}
  
  \noindent
  Since $(f_g, f_g)$ is a refiner for $K$,
  Lemma~\ref{lem-simple-refiner} implies
  that $K \leq \Auto{[\Gamma_{g}]}$.
  We prove that in fact
  $\Auto{[\Gamma_{g}]}=K$, and first note that
   $\Auto{[\Gamma_{g}]} = \Auto{\Gamma_{g}}$.
  Every automorphism of $\Gamma_{g}$
  stabilises the connected components (because
  they have different sizes) and so it induces automorphisms on them.
  Hence
  $\Auto{\Gamma_{g}} \leq \langle (1\,2), (3\,5),
  (3\,6) \rangle$.  But none of the transpositions in
  $\langle (3\,5), (3\,6) \rangle$ is an automorphism of $\Gamma_{g}$, because
  the arcs between $3$, $5$, and $6$ only go in one direction.  Therefore
  $\Auto{\Gamma_{g}} = \< (1\,2), (3\,6\,5) \> = K$, as stated.
\end{example}

\begin{example}[List of subsets stabiliser and
    transporter]\label{ex-perfect-list-of-sets}

  Whenever $k \in \N_{0}$ and $V_{i} \subseteq \Omega$ for each $i \in \{1,
  \ldots, k\}$ and $\mathcal{V} \coloneqq [V_{1}, \ldots, V_{k}]$, we let
  $\Gamma_{\mathcal{V}}$ be the labelled digraph on $\Omega$ without arcs,
  where the label of each vertex $\alpha \in \Omega$ is $\set{i \in \{1,
  \ldots, k\}}{\alpha \in V_{i}}$. For every
  $\stackS \in \Stacks{\Omega}$, define $f_{\mathcal{V}}(\stackS)$ to be the length-one stack
  $[\Gamma_{\mathcal{V}}]$.  If $g \in \Sym{\Omega}$, then
  $\mathcal{V}^{g} \coloneqq [V_{1}^{g}, \ldots, V_{k}^{g}]$.

  Let $\mathcal{V}$ and $\mathcal{W}$ be arbitrary lists of subsets of $\Omega$
  with notation as explained above.
  Then
  $(f_{\mathcal{V}}, f_{\mathcal{W}})$ is a refiner for the set $\set{g
  \in \Sym{\Omega}}{\mathcal{V}^{g} = \mathcal{W}}$,
  and
  $(f_{\mathcal{V}}, f_{\mathcal{V}})$ is a refiner for the group
  $\set{g \in \Sym{\Omega}}{\mathcal{V}^{g} = \mathcal{V}}$.

  To demonstrate this, let
  \(\mathcal{V} = [\{1,3,6\}, \{3,5\}, \{2,4\}, \{2,3,4\}]\) be a list of subsets
  of $\{1,\ldots,6\}$. See
  Figure~\ref{fig-list-of-sets}.
  Since $(f_{\mathcal{V}},f_{\mathcal{V}})$ is a refiner for
  $A \coloneqq \set{g \in \Sn{6}}{\mathcal{V}^{g} = \mathcal{V}}$,
  it follows that $\Auto{[\Gamma_{\mathcal{V}}]}$
  (i.e.,\ $\Auto{\Gamma_{\mathcal{V}}}$) contains $A$; indeed,
  \(A
  =
  \< (1\,6), (2\,4) \>
  =
  \Auto{\Gamma_{\mathcal{V}}}
  \).
\end{example}

\begin{figure}[!ht]
  \begin{center}
    \begin{tikzpicture}
      \foreach \x in {1,2,3,4,5,6} {
        \node[circle, draw=black] (\x) at (-\x*60+120:1.3cm) {$\mathbf{\x}$};};

      \node at (60:2.0cm)   {$\{1\}$};
      \node at (0:2.2cm)    {$\{3,4\}$};
      \node at (-60:2.0cm)  {$\{1,2,4\}$};
      \node at (-120:2.0cm) {$\{3,4\}$};
      \node at (-180:2.0cm) {$\{2\}$};
      \node at (-240:2.0cm) {$\{1\}$};
      \node at (3.5, 0) {};
      \node at (-3.5, 0) {$\Gamma_{\mathcal{V}}$:};
    \end{tikzpicture}
  \end{center}
  \caption{
    The labelled digraph $\Gamma_{\mathcal{V}}$, for $\mathcal{V} \coloneqq [\{1,3,6\}, \{3,5\}, \{2,4\}, \{2,3,4\}]$, from
    Example~\ref{ex-perfect-list-of-sets}.
  }\label{fig-list-of-sets}
\end{figure}

Example~\ref{ex-perfect-list-of-sets} in particular gives
refiners for the stabilisers and transporter sets of ordered partitions.
If we encode a list $[x_{1}, \ldots, x_{m}]$ in $\Omega$ as
$[\{x_{1}\}, \ldots, \{x_{m}\}]$, and a subset $\{y_{1},
\ldots, y_{n}\} \subseteq \Omega$ as $[\{y_{1},
\ldots, y_{n}\}]$, then we see that Example~\ref{ex-perfect-list-of-sets} can be
used to create refiners for the stabilisers and transporters of lists in $\Omega$ or
subsets of $\Omega$.

For unordered partitions, the following example is applicable.

\begin{example}[Refiner for set of subsets stabiliser and
  transporter]\label{ex-set-of-sets}

  Let $\mathcal{V}$ be an arbitrary set of subsets of $\Omega$.
  Let $k \in \N_{0}$ and $V_{i} \subseteq \Omega$ for all $i \in
  \{1, \ldots k\}$ be such that $\mathcal{V} = \{ V_{1}, \ldots, V_{k} \}$.
  We define
  $\Gamma_{\mathcal{V}}$ to be the labelled digraph
  on $\Omega$ whose set of arcs is
  \[
  \set{(\alpha, \beta) \in \Omega \times \Omega}{\alpha \neq \beta\ \text{and}\
  \{\alpha, \beta\} \subseteq V_{i}\ \text{for some}\ i};
  \]
  where the label of a vertex $\alpha$ is
  a list of length $\max\set{|V_{i}|}{i \in \{1, \ldots, k\}}$
  with $i\textsuperscript{th}$ entry
  \[\labelFunc(\alpha)[i] \coloneqq
  (|\set{j \in \{1, \ldots, k\}}{\alpha \in V_{j}\ \text{and}\ |V_{j}| = i}|,
  \,k),\]
  and the label of each arc $(\alpha, \beta)$ in $\Gamma_{\mathcal{V}}$
  is a list of the same length, with
  $i\textsuperscript{th}$ entry
  \[
  \labelFunc(\alpha, \beta)[i] \coloneqq
  (|\set{j \in \{1,\ldots,k\}}{\alpha,\beta \in V_{j}\ \text{and}\ |V_{j}| = i}|,
   \,k).
  \]
  The connected components of $\Gamma_{\mathcal{V}}$ with at least two vertices
  are the sets in $\mathcal{V}$ that are not singletons.
  The label of a vertex (or arc) encodes, for each size of
  subset, the number of subsets in $\mathcal{V}$ that have that size and contain
  that vertex (or arc).

  For every $\stackS \in \Stacks{\Omega}$, we define $f_{\mathcal{V}}(\stackS)$
  to be the length-one stack $[\Gamma_{\mathcal{V}}]$.
  In addition,
  for all $g \in \Sym{\Omega}$,
  we define $\mathcal{V}^{g} = \{V_{1}^{g}, \ldots, V_{k}^{g}\}$.

  Let $\mathcal{V}$ and $\mathcal{W}$ be arbitrary sets of subsets of $\Omega$.
  Since the labelled digraphs $\Gamma_{\mathcal{V}}$ and $\Gamma_{\mathcal{W}}$
  were defined so that
  $\set{g \in \Sym{\Omega}}{\mathcal{V}^{g} = \mathcal{W}} \subseteq
  \Iso{\Gamma_{\mathcal{V}}}{\Gamma_{\mathcal{W}}}$, it follows by
  Lemma~\ref{lem-refiner-nice-condition} that
  $(f_{\mathcal{V}}, f_{\mathcal{W}})$ is a refiner for the set $\set{g
  \in \Sym{\Omega}}{\mathcal{V}^{g} = \mathcal{W}}$;
  in particular,
  $(f_{\mathcal{V}}, f_{\mathcal{V}})$ is a refiner for the stabiliser group
  $\set{g \in \Sym{\Omega}}{\mathcal{V}^{g} = \mathcal{V}}$.

  As a specific example, we
  consider the sets  $\mathcal{V} \coloneqq \{\{1\}, \{1,2,3\}, \{2,4\}\}$
  and $\mathcal{W} \coloneqq \{\{5\}, \{2,3,4\}, \{3,4\}\}$.
  Both $\mathcal{V}$ and $\mathcal{W}$ contain three subsets, which have sizes
  $1$, $2$ and $3$, so it seems superficially plausible that there
  exist elements of $\Sn{5}$ that map $\mathcal{V}$ to $\mathcal{W}$.

  In order
  to search for the set $\set{g \in \Sn{5}}{\mathcal{V}^{g} =
  \mathcal{W}}$,
  then (with all the following notation as defined above) we can
  use the refiner $(f_{\mathcal{V}}, f_{\mathcal{W}})$ to produce
  stacks $[\Gamma_{\mathcal{V}}]$ and $[\Gamma_{\mathcal{W}}]$ such that
  $\Iso{[\Gamma_{\mathcal{V}}]}{[\Gamma_{\mathcal{W}}]}
  = \Iso{\Gamma_{\mathcal{V}}}{\Gamma_{\mathcal{W}}}$ contains this transporter
  set. The labelled digraphs $\Gamma_{\mathcal{V}}$ and $\Gamma_{\mathcal{W}}$
  are depicted in Figure~\ref{fig-set-of-sets};
  although we do not give the correspondence explicitly, two vertices or
  two arcs have the same visual style if and only if they have the same
  label.
  There are many ways to show that $\Gamma_{\mathcal{V}}$ and
  $\Gamma_{\mathcal{W}}$ are non-isomorphic: for example, they have
  different numbers of arcs.  Hence no element of $\Sn{5}$
  maps $\mathcal{V}$ to $\mathcal{W}$.

\begin{figure}[!ht]
  \begin{center}
    \begin{tikzpicture}
      \tikzstyle{white}=[circle, draw=black]
      \tikzstyle{grey}=[circle, draw=black, fill=gray!25]
      \tikzstyle{black}=[circle, draw=white, fill=black!100]
      \tikzstyle{dots}=[circle, draw=black, pattern=dots, pattern color=gray!50]

      \node[style={
            circle,
            draw=black,
            shade,
            bottom color=black!90, top color=gray!10,
            align=center}]
      (1) at (90:1.4cm)   {$\color{white}\mathbf{1}$};
      \node[white]  (2) at (18:1.4cm)   {$\mathbf{2}$};
      \node[grey]   (3) at (-54:1.4cm)  {$\mathbf{3}$};
      \node[style={
            circle,
            draw=black,
            shade,
            top color=gray!100, bottom color=white, align=center}]
      (4) at (-126:1.4cm) {$\color{white}\mathbf{4}$};
      \node[black]  (5) at (162:1.4cm)  {$\color{white}\mathbf{5}$};

      \arcSym{1}{2}
      \arcSym{2}{3}
      \arcSym{1}{3}
      \arcSymDash{2}{4}

      \node at (3, 0) {};
      \node at (-3, 0) {$\Gamma_{\mathcal{V}}$:};
    \end{tikzpicture}
    \quad
    \begin{tikzpicture}
      \tikzstyle{white}=[circle, draw=black]
      \tikzstyle{grey}=[circle, draw=black, fill=gray!25]
      \tikzstyle{black}=[circle, draw=white, fill=black!100]
      \tikzstyle{str-ne}=[circle, draw=black, pattern=grid,
                          pattern color=gray!50]
      \tikzstyle{white}=[circle, draw=black]

      \node[black]  (1) at (90:1.4cm)   {$\color{white}\mathbf{1}$};
      \node[grey]   (2) at (18:1.4cm)   {$\mathbf{2}$};
      \node[white]  (3) at (-54:1.4cm)  {$\mathbf{3}$};
      \node[white]  (4) at (-126:1.4cm) {$\mathbf{4}$};
      \node[str-ne] (5) at (162:1.4cm)  {$\mathbf{5}$};

      \arcSym{2}{4}
      \arcSym{2}{3}
      \arcSymDot{3}{4}

      \node at (3, 0) {};
      \node at (-3, 0) {$\Gamma_{\mathcal{W}}$:};
    \end{tikzpicture}
  \end{center}
  \caption[set-of-sets]{
    Illustration of the labelled digraphs
    $\Gamma_{\mathcal{V}}$ and $\Gamma_{\mathcal{W}}$
    from
    Example~\ref{ex-set-of-sets}, for the sets of subsets
    $\mathcal{V} \coloneqq \{\{1\}, \{1,2,3\}, \{2,4\}\}$
    and $\mathcal{W} \coloneqq \{\{5\}, \{2,3,4\}, \{3,4\}\}$
    of $\{1,\ldots,5\}$.
  }\label{fig-set-of-sets}
\end{figure}
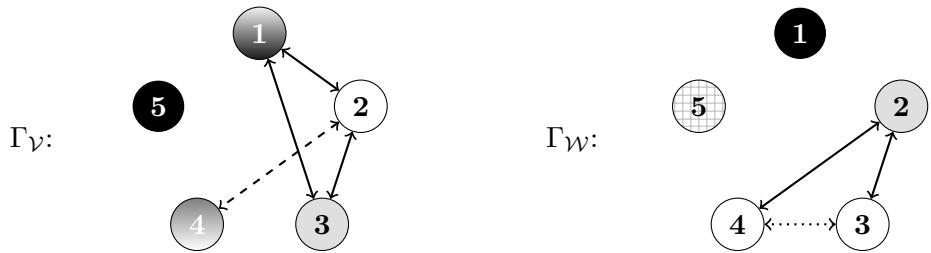
\end{example}

\section{Approximating isomorphisms and fixed points of
  stacks}\label{sec-approximations}

When searching with
labelled digraphs stacks, 
it might be too expensive to compute the set of isomorphisms
exactly, which is why we choose to only approximate this set instead.
Our methods always lead to an overestimation of the set, and 
worse approximations typically lead to larger searches.
As a consequence, there is a compromise to be made
between the accuracy of such
overestimates, and the amount of effort spent in computing them.

In Definition~\ref{defn-approx-iso}, we introduce the concept of an isomorphism
approximator for pairs of labelled digraphs stacks, which is a vital
component of the algorithms in Section~\ref{sec-search}. Later,
we define the approximators that we use in our experiments.

\begin{definition}\label{defn-approx-iso}
  An \emph{isomorphism approximator} for labelled digraph stacks is a function
  $\approxFunc$ that maps a pair of labelled digraph stacks on $\Omega$
  to either the empty set $\varnothing$, or a right coset of a subgroup of
  $\Sym{\Omega}$,
  such that the following statements hold for all $\stackS, \stackT \in
  \Stacks{\Omega}$
  (we usually abbreviate $\approxFunc(\stackS,\stackS)$ as $\approxFunc(\stackS)$):
  \begin{enumerate}[label=\textrm{(\roman*)}]
    \item\label{item-approx-true-overestimate}
      $\Iso{\stackS}{\stackT} \subseteq \approxFunc(\stackS,\stackT)$.

    \item\label{item-approx-different-lengths}
      If $|\stackS| \neq |\stackT|$, then
      $\approxFunc(\stackS,\stackT) = \varnothing$.

    \item\label{item-approx-right-coset-of-aut}
      If $\approxFunc(\stackS,\stackT) \neq \varnothing$, then
      $\approxFunc(\stackS,\stackT) = \approxFunc(\stackS) \cdot h$ for some
      $h \in \Sym{\Omega}$.
  \end{enumerate}
\end{definition}

Let $\approxFunc$ be an isomorphism approximator and let $\stackS, \stackT \in
\Stacks{\Omega}$.
The set  $\Iso{\stackS}{\stackT}$ of isomorphisms induced by $\Sym{\Omega}$ is either empty,
or it is
a right coset of the induced automorphism group $\Auto{\stackS}$.
Since $\idOmega \in \Iso{\stackS}{\stackS} = \Auto{\stackS}$, it follows by
definition that $\approxFunc(\stackS)$ is a subgroup of $\Sym{\Omega}$
that contains $\Auto{\stackS}$.

The value of $\approxFunc(\stackS,\stackT)$ should be interpreted as follows.  By
Definition~\ref{defn-approx-iso}\ref{item-approx-true-overestimate},
$\approxFunc(\stackS,\stackT)$
gives a true overestimate for $\Iso{\stackS}{\stackT}$.  Hence if
$\approxFunc(\stackS,\stackT) = \varnothing$, then the approximator has correctly
determined that $\stackS$ and $\stackT$ are non-isomorphic.
By Definition~\ref{defn-approx-iso}\ref{item-approx-different-lengths},
an isomorphism approximator correctly determines that stacks of different
lengths are non-isomorphic. Otherwise, the approximator returns a right coset
in $\Sym{\Omega}$ of its overestimate for $\Auto{\stackS}$.

In Section~\ref{sec-refiners-via-stack}, we need the ability to
compute fixed points of the automorphism group (induced by $\Sym{\Omega}$) of any 
labelled digraph stack.
A point $\omega \in \Omega$ is \emph{a fixed point of a
subgroup} $G \leq \Sym{\Omega}$ if and only if $\omega^{g}=\omega$ for all $g
\in G$.
Computing fixed points is particularly useful when it comes to using orbits and
orbital graphs in our search techniques.
However, it can be computationally expensive to compute the fixed points exactly,
and so we introduce the following definition.

\begin{definition}\label{defn-approx-fixed}
  A \emph{fixed-point approximator} for labelled digraph stacks is a function
  $\fixedFunc$ that maps each labelled digraph stack on $\Omega$ to
  a finite list in $\Omega$,
  such that for each $\stackS \in \Stacks{\Omega}$:
  \begin{enumerate}[label=\textrm{(\roman*)}]

    \item\label{item-fixed}
      Each entry in \(\Fixed{\stackS}\) is a fixed point of $\Auto{\stackS}$,
      and

    \item\label{item-fixed-invariant}
      \({\Fixed{\stackS}}^{g} = \Fixed{{\stackS}^{g}}\) for all \(g
      \in \Sym{\Omega}\).

  \end{enumerate}
\end{definition}

\subsection{Computing automorphisms and isomorphisms
  exactly}\label{sec-exact-approx}

One way to approximate isomorphisms and fixed points of labelled
digraph stacks is simply to compute them exactly. 
For example, we can convert labelled digraph stacks into their squashed labelled
digraphs in order to take advantage of existing tools for computing with digraphs.

To describe this formally, we require the concept of a canoniser of labelled
digraphs.

\begin{definition}\label{defn-canoniser}
  A \emph{canoniser} of labelled digraphs is a function $\canonFunc$ from
  the set of labelled digraphs on $\Omega$ to $\Sym{\Omega}$ such that, for all
  labelled digraphs $\Gamma$ and $\Delta$ on $\Omega$,
  $\Gamma^{\Canon{\Gamma}} = \Delta^{\Canon{\Delta}}$
  if and only if $\Gamma$ and $\Delta$ are isomorphic.
\end{definition}

We can use the software \bliss~\cite{bliss} or
\nauty~\cite{practical2} to canonise labelled digraphs, after converting
them into vertex-labelled digraphs in a way that preserves isomorphisms.

\begin{definition}[Canonising and computing automorphisms
    exactly]\label{defn-nauty-approx}
  Let $\canonFunc$ be a canoniser of labelled digraphs.
  We define functions \(\fixedCanonical\) and \(\approxCanonical\):
  for all $\stackS, \stackT \in \Stacks{\Omega}$, let $g =
  \Canon{\Squash{\stackS}}$ and $h = \Canon{\Squash{\stackT}}$,
  let $L$ be the list
  $[i \in \Omega\,:\,i \text{\ is fixed by\,}\Auto{\Squash{\stackS}^{g}}]$,
  ordered as in $\Omega$,
  and define
  \begin{align*}
    \fixedCanonical(\stackS)
     & = L^{g^{-1}}, \text{\ and} \\
    \approxCanonical(\stackS,\stackT)
     & =
    \begin{cases}
      \Auto{\Squash{\stackS}} \cdot g h^{-1}
       & \text{if}\ \Squash{\stackS}^{g} = \Squash{\stackT}^{h},    \\
      \varnothing
       & \text{otherwise.}
    \end{cases}
  \end{align*}
\end{definition}

\begin{lemma}\label{lem-nauty-approx}
  Let the functions $\approxCanonical$ and $\fixedCanonical$ be given as in
  Definition~\ref{defn-nauty-approx}.
  Then $\approxCanonical$ is an isomorphism approximator, and
  $\fixedCanonical$ is a fixed-point approximator.
  Furthermore, for all $\stackS, \stackT \in \Stacks{\Omega}$,
  $\approxCanonical(\stackS,\stackT) = \Iso{\stackS}{\stackT}$.
\end{lemma}

\begin{proof}
  Throughout the proof, we repeatedly use Lemma~\ref{lem-squash-same-iso} and
  Definition~\ref{defn-canoniser}. As in Definition~\ref{defn-nauty-approx}, let $g
  = \Canon{\Squash{\stackS}}$ and $h = \Canon{\Squash{\stackT}}$.

  First, we show that $\approxCanonical(\stackS,\stackT) = \Iso{\stackS}{\stackT}$,
  which implies that
  Definition~\ref{defn-approx-iso}\ref{item-approx-true-overestimate}
  and~\ref{item-approx-different-lengths} hold.
  If $\stackS$ and $\stackT$ are non-isomorphic, then $\Squash{\stackS}^{g} \neq
  \Squash{\stackT}^{h}$, and so $\approxCanonical(\stackS,\stackT) =
  \Iso{\stackS}{\stackT} = \varnothing$.
  Otherwise $g h^{-1} \in
  \Iso{\Squash{\stackS}}{\Squash{\stackT}} = \Iso{\stackS}{\stackT}$.
  Therefore
  \[
  \approxCanonical(\stackS,\stackT)
  =
  \Auto{\Squash{\stackS}} \cdot gh^{-1}
  =
  \Auto{\stackS} \cdot gh^{-1}
  =
  \Iso{\stackS}{\stackT}.
  \]
  Definition~\ref{defn-approx-iso}\ref{item-approx-right-coset-of-aut} clearly
  holds.
  Therefore $\approxCanonical$ is an isomorphism approximator.

  Define
  $L = {[i \in \Omega\,:\,i \text{\ is fixed by\,}
        \Auto{\Squash{\stackS}^{g}}]}$,
  ordered as usual in $\Omega$.
  Since
  $
  \Auto{\stackS}^{g}
  =
  \Auto{\Squash{\stackS}}^{g}
  =
  \Auto{\Squash{\stackS}^{g}}
  $,
  it follows that $L$ consists of
  fixed points of $\Auto{\stackS}^{g}$, and so $\fixedCanonical(\stackS)$
  (which equals $L^{g^{-1}}$)
  consists of fixed points of $\Auto{\stackS}$.
  Therefore Definition~\ref{defn-approx-fixed}\ref{item-fixed} holds.  To show
  that Definition~\ref{defn-approx-fixed}\ref{item-fixed-invariant} holds, let
  $x \in \Sym{\Omega}$ be arbitrary and define $r =
  \Canon{\Squash{\stackS^{x}}}$.  Since
  $\Squash{\stackS}$ and $\Squash{\stackS^{x}}$ are isomorphic, it
  follows that $\Squash{\stackS}^{g} = \Squash{\stackS^{x}}^{r}$.
  In particular,
  $g^{-1} x r$ is an automorphism of $\Squash{\stackS}^{g}$,
  and so $g^{-1} x r$ fixes $L$.  Thus
  \[
    {\fixedCanonical(\stackS)}^{x}
    = L^{g^{-1} x}
    = L^{(g^{-1} x r) r^{-1}}
    = L^{r^{-1}}
    = \fixedCanonical(\stackS^{x}). \qedhere
  \]
\end{proof}

\subsection{Approximations via equitable labelled digraphs}\label{sec-equitable}

We can use vertex labels to overestimate the set of isomorphisms from one labelled digraph to another,
because these isomorphisms
map the set of vertices with any particular label onto a set of vertices with the same label.
In this section we use the term \emph{vertex labelling} as an abbreviation for the
restriction of a digraph labelling function to the set of vertices.

In order to present the following approximator functions, we require the
notion of an equitable labelled digraph.

\subsubsection{Equitable labelled digraphs}

\begin{definition}\label{defn-equitable}
  A labelled digraph $(\Omega, A,
    \labelFunc)$ is \emph{equitable} if and only if, for all vertices $\alpha,
    \beta \in \Omega$ with $\labelFunc(\alpha) = \labelFunc(\beta)$, and for all labels $y$ and
    $z$:
  \begin{align*}
     & |\set{(\alpha, \delta) \in A}{\labelFunc(\delta) = y\ \text{and}\
      \labelFunc(\alpha, \delta) = z}|
    = \\ & \hspace{5cm}
    |\set{(\beta,  \delta) \in A}{\labelFunc(\delta) = y\ \text{and}\
      \labelFunc(\beta,  \delta) = z}|,\ \text{and} \\
     & |\set{(\delta, \alpha) \in A}{\labelFunc(\delta) = y\ \text{and}\
      \labelFunc(\delta, \alpha) = z}|
    = \\ & \hspace{5cm}
    |\set{(\delta, \beta) \in A}{\labelFunc(\delta) = y\ \text{and}\
      \labelFunc(\delta, \beta) = z}|.
  \end{align*}
  In other words, the labelled digraph is equitable if and only if, for all
  labels $x$, $y$, and $z$, every vertex with label $x$ has
  some common number of \emph{out-neighbours} with label $y$ via arcs with label
  $z$, and similarly, every vertex with label $x$ has some common number of
  \emph{in-neighbours} with label $y$ via arcs with label $z$.
\end{definition}

Definition~\ref{defn-equitable} extends the well-known
concepts of equitable colourings~\cite[Section~3.1]{practical2} and
partitions~\cite[Definition~29]{newrefiners} of vertex-labelled graphs and
digraphs.
The traditional notion requires that for all labels \(y\) and \(z\),
there are constants for the number of arcs from each vertex with label \(y\) to vertices with label \(z\),
and for the number of arcs in the other direction.
Definition~\ref{defn-equitable} additionally takes arc labels into account.

It is possible to define a procedure that takes a labelled digraph $\Gamma \coloneqq (\Omega, A, \labelFunc)$ and
`refines' the vertex labelling to obtain a labelled digraph $\Gamma' \coloneqq (\Omega, A, \labelFunc')$,
which uses the fewest possible labels such that:
arc labels are unchanged,
vertices with the same label in \(\Gamma'\) have the same label in $\Gamma$,
and \(\Gamma'\) is equitable.
Moreover, this can be done consistently between
labelled digraphs $\Gamma$ and $\Delta$, such that the
overestimate of $\Iso{\Gamma}{\Delta}$ that can be obtained
from the equitable labelled digraphs is contained in the
overestimate from the original vertex labels.
We present an example of such a procedure, phrased as an algorithm, as
Algorithm~4.8 in~\cite{directorscut}.
Here we just describe the idea of such an ``equitable vertex labelling algorithm'' and abbreviate it as \EVLA{}.

Given a labelled digraph, \EVLA{} repeatedly
tests whether each set of vertices with the same label satisfies the condition
in Definition~\ref{defn-equitable}. For each such set, either the
condition is satisfied, and a new label for this set is devised that encodes
information about how the condition was satisfied, or the
condition is not satisfied, and the vertices are given new labels accordingly, which
encode information about why they were created.

We define $\equitableFunc$ to be a function defined by \EVLA{} that
maps each labelled digraph to a list of pairs of the form $(x, W)$, for some
label $x \in \labelSet$ and non-empty $W \subseteq \Omega$, sorted by first component (recall that \(\labelSet\) is totally ordered).
This list encodes that the vertices in $W$ are those with label $x$ in the
equitable digraph given by \EVLA{}.

In the following lemma, we present some properties of $\equitableFunc$.
For a more detailed discussion we refer to~\cite{directorscut}, and we omit the
proof of the lemma because it is mathematically straightforward.

\begin{lemma}\label{lem-equitable}
  Let $\Gamma$ and $\Delta$ be labelled digraphs on $\Omega$,
  and define $k, l \in \N_{0}$, labels $x_{1}, \ldots, x_{k}, y_{1},
  \ldots, y_{l}$, and partitions
  $\{V_{1}, \ldots, V_{k}\}$ and $\{W_{1}, \ldots, W_{l}\}$ of $\Omega$
  such that
  \[
    \Equitable{\Gamma} =
    [(x_{1}, V_{1}), \ldots, (x_{k}, V_{k})]
    \ \text{and}\
    \Equitable{\Delta} =
    [(y_{1}, W_{1}), \ldots, (y_{l}, W_{l})].
  \]
  Then the following hold:
  \begin{enumerate}[label=\textrm{(\roman*)}]
    

    \item\label{item-equitable-map}
      $\Equitable{{\Gamma}^{g}} =
        [(x_{1}, V_{1}^{g}), \ldots, (x_{k}, V_{k}^{g})]$
        for all $g \in \Sym{\Omega}$.

    \item\label{item-equitable-isos}
      $\Iso{\Gamma}{\Delta}
        \begin{cases}
          = \varnothing,
          \quad\text{if}\ k \neq l,
          \ \text{or}\ k = l\ \text{and}\ x_{i} \neq y_{i}
          \ \text{for some}\ i,   \\
          \subseteq
            \set{g \in \Sym{\Omega}}
                {[V_{1}^{g}, \ldots, V_{k}^{g}] = [W_{1}, \ldots, W_{k}]},
           \quad\text{otherwise}.
        \end{cases}$
  \end{enumerate}
\end{lemma}

By choosing meaningful new vertex labels as described, we can distinguish more
pairs of labelled digraphs as non-isomorphic via
Lemma~\ref{lem-equitable}\ref{item-equitable-isos}
than we can by defining new labels arbitrarily.
The next example illustrates this principle.

\begin{example}\label{ex-equitable-iso}
  Let $\Gamma$ be the labelled digraph on $\Omega$ with all possible
  arcs, and let $\Delta$ be the labelled digraph on $\Omega$ without arcs,
  where every vertex and arc in $\Gamma$ and $\Delta$ has the label $x$, for
  some arbitrary but fixed label $x \in \labelSet$.

  Then we may use \EVLA{} to deduce that $\Gamma$ and $\Delta$ are
  non-isomorphic, even
  though both are regular (i.e.\ every vertex has a common number of
  in-neighbours, and a common number of out-neighbours).
  The new labels encode that each vertex in $\Gamma$ and $\Delta$ has $|\Omega|$
  in- and out-neighbours, or zero in- or out-neighbours,
  respectively.  Therefore, the labels given by $\Equitable{\Gamma}$ and
  $\Equitable{\Delta}$ are different, and so $\Gamma$ and $\Delta$ are
  non-isomorphic by Lemma~\ref{lem-equitable}\ref{item-equitable-isos}.

  \textit{A note of warning:} the choice of new labels plays a role!
  If new labels were instead, say, chosen to be incrementally increasing
  integers starting at $1$, then we would have $\Equitable{\Gamma} =
  \Equitable{\Delta}$, and the above deduction would not be possible.
\end{example}

In the previous example it is obvious to us that the digraphs are non-isomorphic,
but for many more complicated
examples, Lemma~\ref{lem-equitable}\ref{item-equitable-isos} can still be used
to detect less obvious non-isomorphism.

\subsubsection{Strong and weak approximations via equitable labelled digraphs}

We describe two strategies for using \EVLA{} (which operates
on labelled digraphs) to approximate isomorphisms and fixed
points of \emph{stacks} of labelled digraphs.
In the first approach, we first combine the entries of a stack into a single
digraph, namely the squashed labelled digraph of the stack, and then apply
\EVLA;
in the other, we first apply \EVLA{} to each of the entries in the stack,
and then combine the information that we obtain.
We call these approaches \emph{strong} and \emph{weak equitable approximation},
respectively, and we give an example of their use in
Section~\ref{sec-approx-comparison}.

\begin{definition}[Strong equitable approximation]\label{defn-strong-approx}
  We define functions \(\approxStrong\) and \(\fixedStrong\) as follows.
  Let $\stackS, \stackT \in \Stacks{\Omega}$. Then there exist
  $k, l \in \N_{0}$, labels $x_{1}, \ldots, x_{k}$,
  $y_{1},
  \ldots, y_{l}$,
  and partitions $\{V_{1}, \ldots, V_{k}\}$ and $\{W_{1}, \ldots, W_{l}\}$
  of $\Omega$ such that
  \begin{align*}
    \Equitable{\Squash{\stackS}} & =
    [(x_{1}, V_{1}), \ldots, (x_{k}, V_{k})],
    \ \text{and}                     \\
    \Equitable{\Squash{\stackT}} & =
    [(y_{1}, W_{1}), \ldots, (y_{l}, W_{l})].
  \end{align*}
  Let $G$ denote the stabiliser of the list $[V_{1}, \ldots, V_{k}]$ in
  $\Sym{\Omega}$,
  and define
  \[
    \approxStrong(\stackS,\stackT) =
    \begin{cases}
      G \cdot h
                  & \text{if\ }
                    |\stackS| = |\stackT|,\
                    k = l,
                    \text{\ and for all}\ i,\,
                    x_{i} = y_{i} \ \text{and}\ |V_{i}| = |W_{i}|; \\
      \varnothing & \text{otherwise,}
    \end{cases}
  \]
  where $h \in \Sym{\Omega}$ is any permutation such that
  $V_{i}^{h} = W_{i}$ for all $i \in \{1,\ldots,k\}$. This is well-defined because, for all $g, h
  \in \Sym{\Omega}$, we have that $V_{i}^{g} = V_{i}^{h}$ for all $i$ if and only if $g$ and
  $h$ represent the same right coset of $G$ in $\Sym{\Omega}$.
  Finally, we define
  \[\fixedStrong(\stackS) = [v_{i_{1}}, \ldots, v_{i_{m}}],\]
  where $i_{1} < \cdots < i_{m}$ and
  the sets $V_{i_{j}} = \{v_{i_{j}}\}$ for each $j \in \{1,\ldots,m\}$ are
  exactly the singletons amongst $V_{1}, \ldots, V_{k}$.
\end{definition} 

\begin{definition}[Weak equitable approximation]\label{defn-weak-approx}
  We define functions \(\approxWeak\) and \(\fixedWeak\) as follows.
  Let $\stackS, \stackT \in \Stacks{\Omega}$.  For each $i \in
    \{1, \ldots, |\stackS|\}$, $j \in \{1, \ldots, |\stackT|\}$,
  there exist $k_{i}, l_{j} \in \N_{0}$, labels $x_{i,1}, \ldots, x_{i,k_{i}},
  y_{j,1}, \ldots, y_{j,l_{j}}$, and partitions
  $\{V_{i,1}, \ldots, V_{i,k_{i}}\}$ and $\{W_{j,1}, \ldots, W_{j,l_{j}}\}$
  of $\Omega$
  such that
  \begin{align*}
    \Equitable{\stackS[i]} & =
    [(x_{i, 1}, V_{i, 1}), \ldots, (x_{i, k_{i}}, V_{i, k_{i}})],
    \ \text{and}               \\
    \Equitable{\stackT[j]} & =
    [(y_{j, 1}, W_{j, 1}), \ldots, (y_{j, l_{j}}, W_{j, l_{j}})].
  \end{align*}

  If $|\stackS| \neq |\stackT|$,
  or else if $k_{i} \neq l_{i}$ for some $i \in \{1, \ldots, |\stackS|\}$,
  or else if $x_{i, j} \neq y_{i, j}$ for some $i \in \{1, \ldots, |\stackS|\}$
  and $j \in \{1, \ldots, k_{i}\}$, then we define
  $\approxWeak(\stackS,\stackT) = \varnothing$.

  Suppose otherwise.
  We define functions $f$ and $g$ that map vertices to lists
  of length \(|S|\) with entries in $\N$.
  For each $\alpha \in \Omega$,
  the list entry $f(\alpha)[i]$
  is the unique $j \in \{1, \ldots, k_{i}\}$ such that $\alpha \in V_{i, j}$,
  and $g(\alpha)[i]$
  is the unique $j \in \{1, \ldots, k_{i}\}$ such that $\alpha \in W_{i, j}$.
  Thus $f$ and $g$ encode the
  `equitable' label of a vertex in each entry of $\stackS$ and $\stackT$, respectively.
  Then we partition $\Omega$ into subsets $A_{1}, \ldots, A_{m}$ according to, and
  ordered lexicographically by, $f$-value, and similarly we partition $\Omega$
  into subsets $B_{1}, \ldots, B_{n}$ via $g$.

  Given all of this, we let $G$ denote the stabiliser of $[A_{1}, \ldots, A_{m}]$
  in $\Sym{\Omega}$ and define
  \[
    \approxWeak(\stackS,\stackT) =
    \begin{cases}
      G \cdot h
                  & \text{if}\
                    |\stackS| = |\stackT|,\ m = n,\ \text{and for all}\ i, \\
                  & \quad |A_{i}| = |B_{i}|\
                    \text{and}\ f(\min(A_{i})) = g(\min(B_{i})), \\
      \varnothing & \text{otherwise,}
    \end{cases}
  \]
  where $h \in \Sym{\Omega}$ is any permutation such that
  $A_{i}^{h} = B_{i}$ for all \(i \in \{1,\ldots,m\}\),
  and $\min(A_{i})$ is the minimum vertex in $A_{i}$
  with respect to the ordering of $\Omega$. 
  We also define
  \[\fixedWeak(\stackS) = [a_{i_{1}}, \ldots, a_{i_{t}}],\]
  where $i_{1} < \cdots < i_{t}$ and the sets $A_{i_{j}} =
  \{a_{i_{j}}\}$ for each $j \in \{1,\ldots,t\}$ are exactly the singletons
  amongst $A_{1}, \ldots, A_{m}$.
\end{definition}

The following lemma holds by Lemma~\ref{lem-equitable}.

\begin{lemma}\label{lem-weak-strong-approx}
  The functions
  from Definitions~\ref{defn-strong-approx} and~\ref{defn-weak-approx},
  $\approxStrong$ and $\approxWeak$,
  and $\fixedStrong$ and $\fixedWeak$,
  are isomorphism and fixed-point approximators,
  respectively.
\end{lemma}

\subsection{Comparing approximators}\label{sec-approx-comparison}

In this section, we give a simple example to compare the isomorphism
approximators from Sections~\ref{sec-exact-approx} and~\ref{sec-equitable}.  We
present the example in more detail in~\cite{directorscut}.

Weak equitable approximations should be the least accurate but cheapest to
compute, whereas computing isomorphisms exactly should be the most
expensive.
Weak equitable approximation distinguishes vertices by
distinguishing them in the individual entries of the stacks.
Strong equitable approximation sometimes
gives better results than this, because it considers the
entire stacks simultaneously.

\begin{example}\label{ex-approx}

  Let $\Gamma_{1}, \Gamma_{2}, \Delta_{1}$, and $\Delta_{2}$ be labelled
  digraphs on $\{1,\ldots,6\}$ whose arcs and arc labels are defined
  as in Figure~\ref{fig-ex-approx}, and where each vertex is labelled
  \exLabel{white}.
  We approximate the isomorphisms from the stack
  $\stackS \coloneqq [\Gamma_{1}, \Gamma_{2}]$ to the stack
  $\stackT \coloneqq [\Delta_{1}, \Delta_{2}]$.

\begin{figure}[!ht]
  \centering
  \begin{tikzpicture}
    \tikzstyle{white}=[circle, draw=black]
    \foreach \x in {1,2,3,4,5,6} {
        \node[white] (\x) at (-\x*60+120:1.5cm) {$\mathbf{\x}$};};
    \foreach \x/\y in {1/2,2/3,3/4,4/5,5/6,6/1} {\arcSym{\x}{\y}};
    \node at (0, -2.0) {$\Gamma_{1}$};
  \end{tikzpicture}
  \qquad\quad
  \begin{tikzpicture}
    \tikzstyle{white}=[circle, draw=black]
    \foreach \x in {1,2,3,4,5,6} {
        \node[white] (\x) at (-\x*60+120:1.5cm) {$\mathbf{\x}$};};
    \foreach \x/\y in {1/2,3/6,4/5} {\arcSymDash{\x}{\y}};
    \node at (0, -2.0) {$\Gamma_{2}$};
  \end{tikzpicture}
  \qquad\quad
  \begin{tikzpicture}
    \tikzstyle{white}=[circle, draw=black]
    \foreach \x in {1,2,3,4,5,6} {
        \node[white] (\x) at (-\x*60+120:1.5cm) {$\mathbf{\x}$};};
    \foreach \x/\y in {2/3,3/4,5/6,6/1} {\arcSym{\x}{\y}};
    \foreach \x/\y in {3/6} {\arcSymDash{\x}{\y}};
    \foreach \x/\y in {1/2,4/5} {\arcSymDot{\x}{\y}};
    \node at (0, -2.0) {$\Squash{[\Gamma_{1}, \Gamma_{2}]}$};
  \end{tikzpicture}

  \vspace{4mm}

  \begin{tikzpicture}
    \tikzstyle{white}=[circle, draw=black]
    \foreach \x in {1,2,3,4,5,6} {
        \node[white] (\x) at (-\x*60+120:1.5cm) {$\mathbf{\x}$};};
    \foreach \x/\y in {6/4,4/5,5/3,3/2,2/1,1/6} {\arcSym{\x}{\y}};
    \node at (0, -2.0) {$\Delta_{1}$};
  \end{tikzpicture}
  \qquad\quad
  \begin{tikzpicture}
    \tikzstyle{white}=[circle, draw=black]
    \foreach \x in {1,2,3,4,5,6} {
        \node[white] (\x) at (-\x*60+120:1.5cm) {$\mathbf{\x}$};};
    \foreach \x/\y in {6/4,5/1,3/2} {\arcSymDash{\x}{\y}};
    \node at (0, -2.0) {$\Delta_{2}$};
  \end{tikzpicture}
  \qquad\quad
  \begin{tikzpicture}
    \tikzstyle{white}=[circle, draw=black]
    \foreach \x in {1,2,3,4,5,6} {
        \node[white] (\x) at (-\x*60+120:1.5cm) {$\mathbf{\x}$};};
    \foreach \x/\y in {4/5,5/3,2/1,1/6} {\arcSym{\x}{\y}};
    \foreach \x/\y in {5/1} {\arcSymDash{\x}{\y}};
    \foreach \x/\y in {6/4,3/2} {\arcSymDot{\x}{\y}};
    \node at (0, -2.0) {$\Squash{[\Delta_{1}, \Delta_{2}]}$};
  \end{tikzpicture}
  \caption[IsoApprox]{\label{fig-ex-approx}
  Pictures of the labelled digraphs on $\{1,\ldots,6\}$ from
  Example~\ref{ex-approx}.
  Each arc in $\Gamma_{1}, \Gamma_{2}, \Delta_{1}$, and $\Delta_{2}$ is
  labelled \exLabel{solid} or \exLabel{dashed} according to its depiction.
  Every vertex in $\Squash{[\Gamma_{1}, \Gamma_{2}]}$ and $\Squash{[\Delta_{1},
  \Delta_{2}]}$ has the same label $[\exLabel{white}, \exLabel{white}]$; arcs
  with label $[\exLabel{solid}, \#]$ are shown as \exLabel{solid}, arcs with
  label $[\#, \exLabel{dashed}]$ are shown as \exLabel{dashed}, and arcs with
  label $[\exLabel{solid}, \exLabel{dashed}]$ are shown as \exLabel{dotted}.
  }
\end{figure}
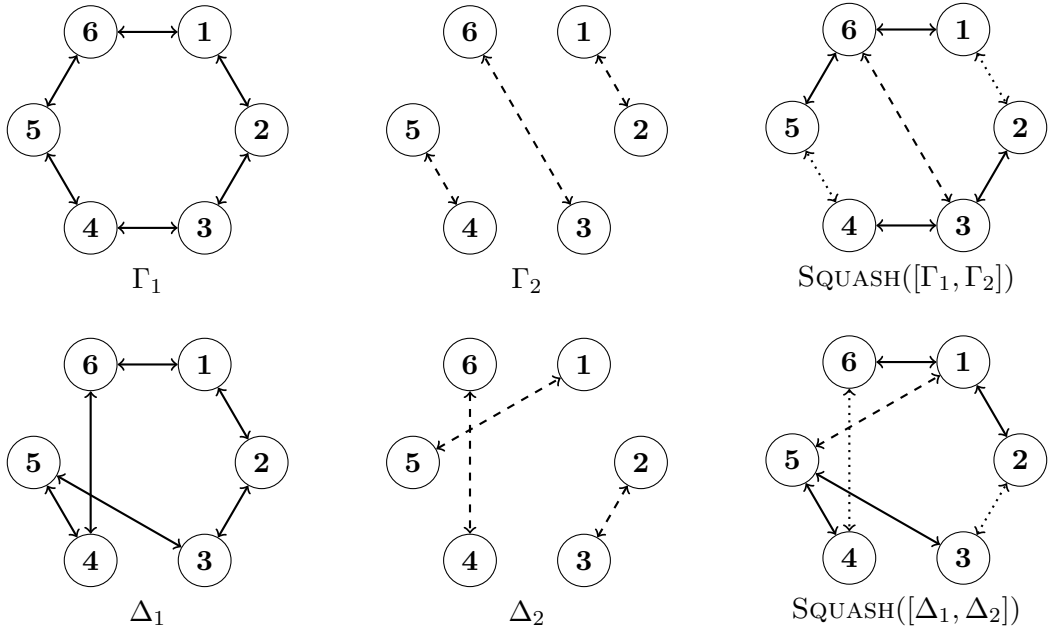

\begin{description}[leftmargin=0mm]
  \item[Weak equitable approximation.]
    The labelled digraphs $\Gamma_{1}, \Gamma_{2},
    \Delta_{1}$, and $\Delta_{2}$ are equitable, and their vertices are all \exLabel{white}.
    Therefore \EVLA{} makes no progress,
    and so weak equitable approximation gives the worst possible result
      \(
        \approxWeak(\stackS,\stackT) = \Sn{6}
      \).
      
  \item[Strong equitable approximation.]
    To see that the labelled digraphs $\Squash{\stackS}$ and
    $\Squash{\stackT}$ are not equitable, note for example that
    there are vertices in each of these digraphs with different numbers of
    out-neighbours, yet all vertices have the same label.
    There exist labels $x$ and $y$ such that the \EVLA{}
    assigns $x$ to $\{3,6\}$ and $y$ to $\{1,2,4,5\}$ in $\Squash{\stackS}$,
    and it assigns $x$ to $\{1,5\}$ and $y$ to $\{2,3,4,6\}$ in
    $\Squash{\stackT}$.

    Let $G$ be the stabiliser of $[\{3,6\}, \{1,2,4,5\}]$ in $\mathcal{S}_{6}$,
    and let $g \in \mathcal{S}_{6}$ be any permutation that maps
    $\{3,6\}$ to $\{1,5\}$ and $\{1,2,4,5\}$ to $\{2,3,4,6\}$.
    Then strong equitable approximation gives
    \(
      \approxStrong(\stackS,\stackT) =
      G \cdot h
    \).
    Note that $|\approxStrong(\stackS,\stackT)| = |G| = 2! \cdot 4! = 48$.

  \item[Canonising and computing exactly.]
    We compute with
    \textsc{Bliss}~\cite{bliss} via
    the \textsc{GAP}~\cite{GAP4} package \textsc{Digraphs}~\cite{digraphs}
    that
    $\Auto{\Squash{\stackS}}
    =
    \< (1\,2)(3\,6)(4\,5), (1\,4)(2\,5)(3\,6) \>
    \eqqcolon
    G$
    and that
    $\Squash{\stackS}^{(1\,2\,3\,5\,6)} =
    \Squash{\stackT}$.
    Thus
    \(
      \Iso{\stackS}{\stackT} =
      G \cdot (1\,2\,3\,5\,6)
    \).
    In particular,
    $|\Iso{\stackS}{\stackT}| = |G| = 4$,
    which reveals the inaccuracy of the other approximators.

\end{description}

\end{example}

\section{Distributing stack isomorphisms across new stacks}\label{sec-splitter}

In a backtrack search, when it is not clear how to further prune a
search space, we divide the search across a number
of smaller areas that can be searched more easily. We call
this process \emph{splitting}, and in this section  we define the
notion of a \emph{splitter} for labelled digraph stacks.
A splitter takes a pair of stacks that represents a (potentially large)
search space, and defines new stacks that divide the space
in a sensible way.

\begin{definition}\label{defn-splitter}
  A \emph{splitter} for an isomorphism approximator $\approxFunc$ is a function
  $\splitFunc$ that maps a pair of labelled digraph stacks on $\Omega$ to a
  finite list of stacks,
  where for all $\stackS, \stackT \in \Stacks{\Omega}$
  with $|\Approx{\stackS}{\stackT}| \geq 2$,
  \[
    \Split{\stackS}{\stackT}
    =
    [
    \stackS_{1},
    \stackT_{1},
    \stackT_{2},
    \ldots,
    \stackT_{m}
    ]
  \]
  for some $m \in \N_{0}$
  and $\stackS_{1}, \stackT_{1}, \ldots, \stackT_{m} \in \Stacks{\Omega}$, such that:
  \begin{enumerate}[label=\textrm{(\roman*)}]
    \item\label{item-splitter-union}
      $\Iso{\stackS}{\stackT}
        =
        \Iso{\stackS \mathop{\Vert} \stackS_{1}}{\stackT \mathop{\Vert} \stackT_{1}}
        \mathop{\cup}
        \cdots
        \mathop{\cup}
        \Iso{\stackS \mathop{\Vert} \stackS_{1}}{\stackT \mathop{\Vert} \stackT_{m}}$.

      \item\label{item-splitter-smaller}
      $|\Approx{\stackS \mathop{\Vert} \stackS_{1}}{\stackT \mathop{\Vert} \stackT_{i}}|
       <
       |\Approx{\stackS}{\stackT}|$ for all $i \in \{1, \ldots, m\}$.

    \item\label{item-splitter-invariant}
      For all $\stackU \in \Stacks{\Omega}$ with $|\Approx{\stackS}{\stackU}| \geq 2$,
      the first entry of \(\Split{\stackS}{\stackU}\) is \(S_1\).

  \end{enumerate}
\end{definition}

For this paragraph and the following remark, we keep the notation from Definition~\ref{defn-splitter}, with
$|\Approx{\stackS}{\stackT}| \geq 2$.
The search space corresponding to the pair $(\stackS,\stackT)$ is
$\Approx{\stackS}{\stackT}$.
By
Definition~\ref{defn-splitter}\ref{item-splitter-union}, if $|\Approx{\stackS}{\stackT}| \geq 2$,
then the splitter produces the search spaces
$\Approx{\stackS \mathop{\Vert} \stackS_{1}}{\stackT \mathop{\Vert} \stackT_{i}}$
for each $i \in \{1,\ldots,m\}$;
note that the left-hand stack \(\stackS \mathop{\Vert} \stackS_{1}\) does not
vary here.
Each one of these new search spaces is smaller than
$\Approx{\stackS}{\stackT}$, by
Definition~\ref{defn-splitter}\ref{item-splitter-smaller}.  This is required to
show that our algorithms terminate.  
Definition~\ref{defn-splitter}\ref{item-splitter-invariant} means that the first stack
given by a splitter is independent of the given right-hand stack.
This is required by the technique in Section~\ref{sec-r-base}.

\begin{remark}\label{rmk-splitter-first}
  If \(S = T\), then it follows by Definition~\ref{defn-splitter}\ref{item-splitter-union} that \(S_1 = T_i\) for some \(i\).
  Thus we may assume without loss of generality that \(S_1 = T_1\) in this case.
\end{remark}

The following lemma shows a way of giving a splitter by specifying its behaviour
on the left stack that it is given.
The proof is straightforward and therefore omitted; see~\cite{directorscut}.

\begin{lemma}\label{lem-splitter-creator}
  Let $\approxFunc$ be an isomorphism approximator, and let $f$ be any
  function from $\Stacks{\Omega}$ to itself such that, for all
  $\stackS \in \Stacks{\Omega}$:
  \[\text{if}\ |\Approx{\stackS}{}| \geq 2,
    \ \text{then}\
    |\Approx{\stackS \mathop{\Vert} f(\stackS)}{}|
    <
    |\Approx{\stackS}{}|.
  \]
  Let $\stackS, \stackT \in \Stacks{\Omega}$,
  and fix an ordering $\stackT_{1}, \ldots, \stackT_{m}$,
  for some $m \in \N_{0}$, of the set
  $\set{{f(\stackS)}^{g}}{g \in \Approx{\stackS}{\stackT}}$.
  Finally, let
  $\splitFunc_{f}(\stackS,\stackT) = [f(\stackS), \stackT_{1}, \ldots, 
  \stackT_{m}]$.
  Then $\splitFunc$ is a splitter for $\approxFunc$.
\end{lemma}

Let the notation of Lemma~\ref{lem-splitter-creator} hold.
Splitting by appending the stack $f(\stackS)$
to the stack $\stackS$ corresponds to stabilising $f(\stackS)$ in the current
approximation of $\Auto{\stackS}$; the stacks of the form
$\stackT_{i}$ give the images of the stack $f(\stackS)$ under
$\Approx{\stackS}{\stackT}$.

Note that the set $\set{{f(\stackS)}^{g}}{g \in \Approx{\stackS}{\stackT}}$
can be computed via the orbit of $f(\stackS)$ under
$\Approx{\stackS}{}$.
Indeed, if $h \in \Approx{\stackS}{\stackT}$, then
since
$\Approx{\stackS}{\stackT} = \Approx{\stackS}{} \cdot h$
by Definition~\ref{defn-approx-iso}\ref{item-approx-right-coset-of-aut}, we have
\begin{align*}
  \set{{f(\stackS)}^{g}}{g \in \Approx{\stackS}{\stackT}}
  & =
  \set{{f(\stackS)}^{g}}{g \in \Approx{\stackS}{} \cdot h} \\
  & =
  \set{{f(\stackS)}^{x}}{x \in \Approx{\stackS}{}}^{h}
  =
  {\left({f(\stackS)}^{\Approx{\stackS}{}}\right)}^{h}.
\end{align*}

In the following definition, we present a splitter that can
be obtained with Lemma~\ref{lem-splitter-creator}.
We use a version of this splitter for our experiments in
Section~\ref{sec-experiments}.
\begin{definition}[Fixed point splitter]\label{defn-point-splitter}
  For all $\alpha \in \Omega$, let $\Gamma_{\alpha} = (\Omega,
  \varnothing, \labelFunc)$ be the labelled digraph on $\Omega$ where
  $\labelFunc(\alpha) =
  1$ and $\labelFunc(\beta) = 0$ for all $\beta \in \Omega \setminus
  \{\alpha\}$.
  Note that $\Gamma_{\alpha}^{g} = \Gamma_{\alpha^{g}}$ for all $g \in
  \Sym{\Omega}$.
  Let $\approxFunc$ be any isomorphism approximator such that
  $\Approx{\stackU \mathop{\Vert} [\Gamma_{\alpha}]}{}
  \leq \Approx{\stackU}{} \cap \set{g \in \Sym{\Omega}}{\alpha^{g} =
  \alpha}$
  for all $\alpha \in \Omega$ and $\stackU \in \Stacks{\Omega}$.
  We define a function $\sigma$ from $\Stacks{\Omega}$ to itself by
  \[
    \sigma(\stackS) =
    \begin{cases}
        \EmptyStack{\Omega}
        &
        \text{if}\ |\Approx{\stackS}{}| \leq 1, \\
        [\Gamma_{\alpha}]
        &
        \text{otherwise, where} \
        \alpha \coloneqq
        \min\{\min(\mathcal{O})\,:\, \mathcal{O}\ \text{is an orbit of} \\
        &
        \qquad
        \Approx{\stackS}{}\
        \text{of minimal size, subject to}\ 
        |\mathcal{O}| \geq 2 \}.
    \end{cases}
  \]
  for all $\stackS \in \Stacks{\Omega}$.
  Finally, define $\splitFunc_{\sigma}$ as in Lemma~\ref{lem-splitter-creator},
  for the isomorphism approximator $\approxFunc$ and the function $\sigma$.
\end{definition}

The following corollary holds by Lemma~\ref{lem-splitter-creator}, bearing in mind that
the function $\sigma$ has the required property by our
choice of isomorphism approximator in
Definition~\ref{defn-point-splitter}.

\begin{cor}\label{cor-point-splitter}
  The function $\splitFunc_{\sigma}$ from Definition~\ref{defn-point-splitter} is a splitter for
  any isomorphism approximator that satisfies the condition in
  Definition~\ref{defn-point-splitter}.
\end{cor}

\section{The search algorithm}\label{sec-search}

Let $U_{1}, \ldots, U_{k} \subseteq \Sym{\Omega}$.
In this section, we present our main algorithms, which combine the tools of
Sections~\ref{sec-stacks}--\ref{sec-splitter} to compute
the intersection $U_{1} \cap \cdots \cap U_{k}$.
In Section~\ref{sec-basic-search},
we show how to perform a backtrack search
for one or all of the elements of $U_{1} \cap \cdots \cap U_{k}$.
In Section~\ref{sec-generators},
when the result is known to form a group,
we show how to search for a base
and strong generating set instead
(see~\cite[p.~101]{dixonmortimer} for a definition).
We explain these algorithms in further detail
in~\cite{directorscut}.
A version of our algorithms is implemented in the
\textsc{GraphBacktracking} package~\cite{GAPpkg} for \textsc{GAP}~\cite{GAP4}.

\subsection{The basic method}\label{sec-basic-search}

We begin with a high-level description of Algorithm~\ref{alg-basic-search},
which comprises the \Call{Search}{} and \Call{Refine}{} procedures.
We say that an algorithm \emph{backtracks} when it finishes executing a
recursive call to a procedure, and continues executing from
where the call was initiated.

\begin{algorithm}[!ht]
  \caption{A recursive algorithm using labelled digraph stacks to search
    in $\Sym{\Omega}$.}\label{alg-basic-search}

  \begin{algorithmic}[1]

    \item[\textbf{Input:}]
    a sequence of subsets
    $U_{1}, \ldots, U_{k} \subseteq \Sym{\Omega}$;

    a sequence $(f_{L,1}, f_{R,1}), \ldots, (f_{L,m}, f_{R,m})$, where
    each pair is a refiner for some $U_{j}$;

    an isomorphism approximator $\approxFunc$ and a splitter $\textsc{Split}$
    for $\approxFunc$.

    \item[\textbf{Output:}]
    all elements of the intersection $U_{1} \cap \cdots \cap U_{k}$,
    which we refer to as \emph{solutions}.

    \vspace{2mm}
    \Procedure{Search}{$\stackS, \stackT$}
    \Comment{The main recursive search procedure
    (Lemma~\ref{lem-alg-basic-search-is-correct}).}\label{line-search-proc-start}

    \State{\((\stackS, \stackT) \gets \Call{Refine}{\stackS, \stackT}\)}
    \Comment{Refine the given stacks.}\label{line-refine}

    \Case{$\Approx{\stackS}{\stackT} = \varnothing$}
      \Comment{Nothing found in the present branch: backtrack.}\label{line-case-empty}
      \State{\Return{\(\varnothing\)}}\label{line-no-isos}
    \EndCase{}

    \Case{$\Approx{\stackS}{\stackT} = \{h\}$ for some $h$}
      \Comment{$h$ is the sole potential solution
      here.}\label{line-case-one-candidate}
      \If{
          $\stackS^{h} = \stackT$ and
          $h \in U_{1} \cap \cdots \cap U_{k}$}\label{line-check-perm}
      \State{\Return{$\{h\}$}}
      \Comment{$h$ is the unique solution in
        $\Iso{\stackS}{\stackT}$: backtrack.}\label{line-single-solution}
      \Else{}
      \State{\Return{$\varnothing$}}
      \Comment{$h$ is not a solution in
        $\Iso{\stackS}{\stackT}$: backtrack.}\label{line-not-solution}
      \EndIf{}
    \EndCase{}

    \Case{$|\Approx{\stackS}{\stackT}| \geq 2$}
      \Comment{Multiple potential solutions.}\label{line-multiple-candidates}

      \State{\([S_1, T_1, \ldots, T_t] \gets \Call{Split}{\stackS, \stackT}\)}
      \Comment{Split the search space.}\label{line-split}

      \State{\Return{\(
        \displaystyle\bigcup_{i \in \{1,\ldots,t\}}
        \Call{Search}{\stackS \mathop{\Vert} \stackS_{1}, \stackT \mathop{\Vert} \stackT_{i}}\)}}
      \Comment{Search recursively.}\label{line-recurse}
    \EndCase{}

    \EndProcedure{}

    \vspace{1mm}
    \Procedure{Refine}{$\stackS, \stackT$}
    \Comment{Attempt to prune the search space
    (Lemma~\ref{lem-alg-refine-is-correct}).}\label{line-refine-proc-start}

    \While{\(\Approx{\stackS}{\stackT} \neq
    \varnothing\)}
    \Comment{Proceed while there are potential
    solutions.}\label{line-while-begin}\label{line-refine-empty}

    \State{\((\stackS', \stackT') \gets
      (\stackS, \stackT)\)}
    \Comment{Save the stacks before the next round of
    refining.}\label{line-refine-store}

    \For{\(i \in \{1, \ldots, m\}\) \textbf{and while} \(|\stackS| =
    |\stackT|\)}\label{line-refine-loop}
      \State{$(\stackS, \stackT) \gets
          (\stackS \mathop{\Vert} f_{L,i}(\stackS),
          \stackT \mathop{\Vert} f_{R,i}(\stackT))$}
      \Comment{Apply each refiner in turn.}\label{line-apply-refiner}
    \EndFor{}

    \If{\(|\Approx{\stackS}{\stackT}| \not<
      |\Approx{\stackS'}{\stackT'}|\)}\label{line-refine-not-smaller}

      \State{\Return{\((\stackS', \stackT')\)}}
      \Comment{Stop: the last refinements seemingly made no
        progress.}\label{line-refine-loop-end}

    \EndIf{}

    \EndWhile{}

    \State{\Return{\((\stackS, \stackT)\)}}\Comment{Stop:
    $\Approx{\stackS}{\stackT} = \varnothing$: no solutions in this
    branch.}\label{line-refine-return-empty}

    \EndProcedure{}

    \vspace{1mm}
    \State{\Return{\Call{Search}{$\EmptyStack{\Omega},
            \EmptyStack{\Omega}$}}}\label{line-start-search}

  \end{algorithmic}
\end{algorithm}

The algorithm begins with a call to the \Call{Search}{} procedure
on line~\ref{line-start-search}.
This procedure,
when given labelled digraph stacks $\stackS$ and $\stackT$,
finds those elements of 
$U_{1} \cap \cdots \cap U_{k}$
that induce isomorphisms from $\stackS$ to $\stackT$
(Lemma~\ref{lem-alg-basic-search-is-correct}).
It does so by searching in
$\Approx{\stackS}{\stackT}$ rather than in \(\Iso{\stackS}{\stackT}\), because
we do not necessarily wish to compute \(\Iso{\stackS}{\stackT}\) exactly.

The \Call{Search}{} procedure first calls \Call{Refine}{}.
This applies the refiners in turn,
aiming to prune the search space
(Lemma~\ref{lem-alg-refine-is-correct}).
Then, if the remaining search space contains at most one element, the
\Call{Search}{} procedure backtracks, having potentially returned an element,
if appropriate.
Otherwise, it divides the search with a splitter,
and recurses.

We assume the given approximator, splitter, and refiners are possible to
compute.

We thus claim that, given the specified inputs, and after a
finite number of steps, Algorithm~\ref{alg-basic-search} returns
the intersection $U_{1} \cap \cdots \cap U_{k}$:

\begin{theorem}\label{thm-alg-is-correct}
  Algorithm~\ref{alg-basic-search} is correct.
\end{theorem}

The proof of Theorem~\ref{thm-alg-is-correct} relies on the following lemmas.
In particular, it follows from Lemma~\ref{lem-alg-basic-search-is-correct} by
setting $\stackS = \stackT = \EmptyStack{\Omega}$.

\begin{lemma}\label{lem-alg-refine-is-correct}
  Let the notation of Algorithm~\ref{alg-basic-search} hold.
  Then the \Call{Refine}{} procedure terminates after a finite number of steps,
  with
  \begin{enumerate}[label=\textrm{(\roman*)}]
    \item\label{item-refine-gives-smaller-approx}
      \(
      |\Approx{\Call{Refine}{\stackS, \stackT}}{}|
      \leq
      |\Approx{\stackS}{\stackT}|
      \), and
    \item\label{item-refine-is-correct}
      \(
      (U_{1} \cap \cdots \cap U_{k})
        \cap \Iso{\stackS}{\stackT}
      =
      (U_{1} \cap \cdots \cap U_{k})
        \cap \Iso{\Call{Refine}{\stackS, \stackT}}{}\).
  \end{enumerate}
\end{lemma}

\begin{proof}
  The \Call{Refine}{} procedure performs finitely
  many iterations of its \emph{while} loop,
  and so terminates,
  since a new iteration occurs only if
  the prior one yields a smaller search space.
  It is evident in the definition of \Call{Refine}{}
  that~\ref{item-refine-gives-smaller-approx} holds.
  To prove~\ref{item-refine-is-correct},
  note that $\Call{Refine}{\stackS, \stackT}$ is obtained
  from $(\stackS,\stackT)$
  by the repeated application of refiners to stacks of equal length
  (line~\ref{line-apply-refiner}).
  Thus it suffices to show that if $i \in \{1, \ldots, m\}$ and
  $|\stackS| = |\stackT|$, then
  \begin{equation*}\label{eq-alg-refine-is-correct}
    (U_{1} \cap \cdots \cap U_{k})
      \cap
    \Iso{\stackS}{\stackT}
    =
    (U_{1} \cap \cdots \cap U_{k})
      \cap
    \Iso{\stackS \mathop{\Vert} f_{L,i}(\stackS)}{\stackT \mathop{\Vert} f_{R,i}(\stackT)}.
  \end{equation*}
  If $\Iso{\stackS}{\stackT} = \varnothing$, then also
  $\Iso{\stackS \mathop{\Vert} f_{L,i}(\stackS)}{\stackT \mathop{\Vert}
  f_{R,i}(\stackT)} = \varnothing$
  by Remark~\ref{rmk-stack-iso-auto}, thereby satisfying the equation.
  Otherwise, the equation holds by
  Lemma~\ref{lem-refiner-equiv-definitions}\ref{item-refiner-long}.
\end{proof}

\begin{lemma}\label{lem-alg-basic-search-is-correct}
  Let the notation of Algorithm~\ref{alg-basic-search} hold.
  Then the \Call{Search}{} procedure terminates with
  \(
  \Call{Search}{\stackS, \stackT}
  =
  (U_{1} \cap \cdots \cap U_{k}) \cap \Iso{\stackS}{\stackT}
  \).
\end{lemma}

\begin{proof}
  We proceed by induction on 
  $|\Approx{\Call{Refine}{\stackS, \stackT}}{}|$,
  bearing in mind
  Definition~\ref{defn-approx-iso}\ref{item-approx-true-overestimate}
  and Lemma~\ref{lem-alg-refine-is-correct}\ref{item-refine-is-correct}.
  If $|\Approx{\Call{Refine}{\stackS, \stackT}}{}| \in \{0,1\}$,
  then it is straightforward to verify that the \Call{Search}{} procedure
  terminates with the correct value.

  Let $n \in \N$ with $n \geq 2$, and assume the statement holds
  for all stacks $\stackS$ and $\stackT$ with
  $|\Approx{\Call{Refine}{\stackS, \stackT}}{}| < n$.
  If $|\Approx{\Call{Refine}{\stackS, \stackT}}{}| = n$, then
  on line~\ref{line-split}, the splitter
  gives a finite list
  of stacks, giving a finite number of calls to \Call{Search}{}.
  By Definition~\ref{defn-splitter}\ref{item-splitter-union}
  and~\ref{item-splitter-smaller},
  by Lemma~\ref{lem-alg-refine-is-correct}\ref{item-refine-gives-smaller-approx},
  and by assumption, these recursive calls terminate
  with values whose union is
  \((U_{1} \cap \cdots \cap U_{k}) \cap \Iso{\Call{Refine}{\stackS,
  \stackT}}{}\).
  The result follows by
  Lemma~\ref{lem-alg-refine-is-correct}\ref{item-refine-is-correct}, and by induction.
\end{proof}

\begin{remark}\label{rmk-search-single}
Algorithm~\ref{alg-basic-search} finds \emph{all} elements
of $U_{1} \cap \cdots \cap U_{k}$.
To search for a \emph{single} element, if one exists,
one can modify the \Call{Search}{} procedure to return a result on
line~\ref{line-single-solution}, as soon as the first solution
is found.
We name this modified procedure \Call{SearchSingle}{}.
Thus $\Call{SearchSingle}{\stackS,\stackT}$ gives a single element of
$\Iso{\stackS}{\stackT} \cap (U_{1} \cap \cdots \cap U_{k})$,
if one exists, else $\varnothing$.
This is especially useful when one wishes to find an
isomorphism from one combinatorial structure to another, or to prove that they
are non-isomorphic.
\end{remark}

\subsection{Searching for a generating set of a subgroup}\label{sec-generators}

It is usually most efficient to compute with a permutation group
via a base and strong generating set (BSGS).
In the standard definition (see~\cite[p.~101]{dixonmortimer}),
a base for a subgroup $G$ of $\Sym{\Omega}$
is a list of points in $\Omega$ whose stabiliser in $G$ is trivial.
Here, we use a broader definition, where the list may contain \emph{any} objects
on which $G$ acts.

In this section
we present Algorithm~\ref{alg-bsgs-search}, which,
given subsets $U_{1}, \ldots, U_{k} \subseteq \Sym{\Omega}$ whose
intersection is a subgroup,
returns a base and strong generating set for $U_{1} \cap \cdots \cap U_{k}$.
The base is given as a list of labelled digraph stacks.
This algorithm is derived from Algorithm~\ref{alg-basic-search}:
the first two cases simplify since \(\Approx{S}{S}\) is a subgroup,
and the recursive case turns into a search for a stabiliser and coset representatives.
Note that Algorithm~\ref{alg-bsgs-search} uses the partially-constructed
generating set to prune the search on line~\ref{line-pruning}.
It thus usually performs a smaller search than does
Algorithm~\ref{alg-basic-search} with the same input.

\begin{algorithm}[!ht]
  \caption{Search for a base and strong generating set of a
           subgroup of $\Sym{\Omega}$.}\label{alg-bsgs-search}

  \begin{algorithmic}[1]

    \item[\textbf{Input:}]
      as in Algorithm~\ref{alg-basic-search}, plus the assumption that
      $U_{1} \cap \cdots \cap U_{k}$ is a subgroup.
    \item[\textbf{Output:}]
      a base and strong generating set of the subgroup $U_{1} \cap \cdots \cap
      U_{k}$.

    \vspace{2mm}

    \Procedure{SearchBSGS}{$\stackS$}\Comment{See
    Lemma~\ref{lem-SearchBSGS-is-correct}.}

    \State{\((\stackS, \stackS) \gets \Call{Refine}{\stackS, \stackS}\)}
    \Comment{Refine the given stacks.}\label{line-gens-refine}

    \Case{$\Approx{\stackS}{\stackS} = \{\idOmega\}$}\label{line-bsgs-base-case}
      \State{\Return{$\left( \textsc{EmptyBase}, \{\idOmega\} \right)$}}
      \Comment{A BSGS for the trivial group.}
    \EndCase{}

    \Case{$|\Approx{\stackS}{\stackS}| \geq 2$}\label{line-bsgs-recursive-case}

      \State{$[\stackS_{1}, \stackS_{1}, \ldots, \stackS_{t}] \gets \Call{Split}{\stackS, \stackS}$}
      \Comment{See Remark~\ref{rmk-splitter-first}.}\label{line-genset-split}

      \State{$\left( \textsc{Base}, X \right) \gets
              \Call{SearchBSGS}{\stackS \mathop{\Vert} \stackS_{1}}$}
              \hspace{-2mm}
      \Comment{Find BSGS for the stabiliser of $\stackS_{1}$.}\label{line-gens-recurse}

      \State{$\textsc{Base} \gets [\stackS_{1}] \mathop{\Vert} \textsc{Base}$}
      \Comment{Prepend $\stackS_{1}$ to the base for the stabiliser of
      $\stackS_{1}$.}

      \For{$i \in \{2, \ldots, t\}$}

        \If{$\stackS_{i} \not\in \stackS_{j}^{\<X\>}$ for any
            $j \in \{1, \ldots, i - 1\}$}
            \Comment{Pruning.}\label{line-pruning}

          \State{$X \gets X \cup
                   \Call{SearchSingle}{\stackS \mathop{\Vert} \stackS_{1},
                                       \stackS \mathop{\Vert} \stackS_{i}}$}
          \Comment{Search for a coset rep.}\label{line-search-for-coset-rep}
        \EndIf{}
      \EndFor{}

      \State{\Return{$\left( \textsc{Base}, X \right)$}}\label{line-genset-return}
    \EndCase{}

    \EndProcedure{}

    \Procedure{Refine}{$\stackS, \stackT$}
    \Comment{\emph{The \Call{Refine}{} procedure from
    Algorithm~\ref{alg-basic-search}.}}
    \EndProcedure{}

    \Procedure{SearchSingle}{$\stackS, \stackT$}
    \Comment{\emph{The procedure from
    Remark~\ref{rmk-search-single}.}}
    \EndProcedure{}

    \State{\Return{\Call{SearchBSGS}{$\EmptyStack{\Omega}$}}}\label{line-start-search-bsgs}

  \end{algorithmic}
\end{algorithm}

\begin{remark}
To obtain a base consisting of points in $\Omega$,
one can use the splitter from
Definition~\ref{defn-point-splitter}: using its notation,
Algorithm~\ref{alg-bsgs-search} returns
a base $[[\Gamma_{\alpha_{1}}], \ldots, [\Gamma_{\alpha_{r}}]]$.
For each $\alpha \in \Omega$, a permutation fixes
$[\Gamma_{\alpha}]$ if and only if it fixes $\alpha$, so
a generating set is strong with respect to
$[[\Gamma_{\alpha_{1}}], \ldots, [\Gamma_{\alpha_{r}}]]$ if and only if it is
strong with respect to $[\alpha_{1}, \ldots, \alpha_{r}]$.
\end{remark}

\begin{remark}
Algorithm~\ref{alg-bsgs-search} is useful when searching for an
intersection of cosets.
Let the notation of Algorithm~\ref{alg-bsgs-search} hold, and suppose that
each set $U_{i}$ is a coset of a subgroup of $\Sym{\Omega}$.
We can use the \Call{SearchSingle}{} procedure
to find some $g \in
U_{1} \cap \cdots \cap U_{k}$, or to prove that no such element exists.
In the former case, then for all $i \in \{1, \ldots, k\}$,
$(f_{L,i},f_{L,i})$ is a refiner for the group $U_{i} g^{-1}$
by Lemma~\ref{lem-refiner-right-coset}.
Therefore we may use Algorithm~\ref{alg-bsgs-search} to search for a
base and strong generating set of $U_{1} g^{-1} \cap \cdots \cap U_{k} g^{-1}$,
which in combination with $g$ compactly describes
$U_{1} \cap \cdots \cap U_{k}$.
\end{remark}

\begin{lemma}\label{lem-SearchBSGS-is-correct}
  Let the notation of Algorithm~\ref{alg-bsgs-search} hold.
  Then the \Call{SearchBSGS}{} procedure, given $\stackS$, terminates with
  a base and strong generating set of
  \(
  \Auto{\stackS}
  \cap
  (U_{1} \cap \cdots \cap U_{k})
  \).
\end{lemma}

\begin{proof}
The \Call{SearchBSGS}{} procedure
first calls the \Call{Refine}{} procedure to prune the search
(see Lemma~\ref{lem-alg-refine-is-correct}).
This returns a pair of equal stacks,
since it is given equal stacks, and each refiner is a
pair of equal functions (see Lemma~\ref{lem-group-refiner-symmetric}).
That one of the conditions on either line~\ref{line-bsgs-base-case}
or line~\ref{line-bsgs-recursive-case} is satisfied
follows by Lemma~\ref{defn-approx-iso}\ref{item-approx-true-overestimate}.
The procedure accordingly returns a trivial solution, or searches
recursively.

The \Call{SearchBSGS}{} procedure only differs significantly from the
\Call{Search}{} procedure of Algorithm~\ref{alg-basic-search} in its recursive
step.
Let
$\Call{Split}{\stackS, \stackS} =
[\stackS_{1}, \stackS_{1}, \stackS_{2}, \ldots, \stackS_{t}]$
as on line~\ref{line-genset-split}.
Then $\Auto{\stackS \mathop{\Vert} \stackS_{1}}$
is the stabiliser of $\stackS_{1}$ in
$\Auto{\stackS}$
by Remark~\ref{rmk-stack-iso-auto},
and
by Definition~\ref{defn-splitter} and Remark~\ref{rmk-splitter-first},
its right cosets in
$\Auto{\stackS}$
are the non-empty sets
$\Iso{\stackS \mathop{\Vert} \stackS_{1}}{\stackS \mathop{\Vert}
\stackS_{i}}$
for $i \in \{2,\ldots,t\}$.
Analogous statements hold for
the stabiliser 
$\Auto{\stackS \mathop{\Vert} \stackS_{1}} \cap (U_{1} \cap \cdots \cap U_{k})$
of $\stackS_{1}$ in
$\Auto{\stackS} \cap (U_{1} \cap \cdots \cap U_{k})$,
and for the sets
$\Iso{\stackS \mathop{\Vert} \stackS_{1}}
     {\stackS \mathop{\Vert} \stackS_{i}}
     \cap
     (U_{1} \cap \cdots \cap U_{k})$.
It is thus possible to build a BSGS of
$\Auto{\stackS} \cap (U_{1} \cap \cdots \cap U_{k})$
recursively from a BSGS of the stabiliser
of $\stackS_{1}$ in $\Auto{\stackS} \cap (U_{1} \cap \cdots \cap U_{k})$,
along with representatives of a sufficient selection of the right cosets of the
stabiliser,
as is done in lines~\ref{line-gens-recurse}--\ref{line-search-for-coset-rep}.

The validity of the recursion in the \Call{SearchBSGS}{} procedure can be shown
by induction on $|\Approx{\Call{Refine}{\stackS, \stackS}}{}|$,
as in the proof of Lemma~\ref{lem-alg-basic-search-is-correct}.
It follows by Remark~\ref{rmk-search-single} that, on
line~\ref{line-search-for-coset-rep},
the \Call{SearchSingle}{} procedure returns the desired representatives
of the sets
$\Iso{\stackS \mathop{\Vert} \stackS_{1}}
     {\stackS \mathop{\Vert} \stackS_{i}}
     \cap
     (U_{1} \cap \cdots \cap U_{k})$.
\end{proof}

\begin{theorem}\label{thm-alg-gens-is-correct}
  Algorithm~\ref{alg-bsgs-search} is correct.
\end{theorem}

\begin{proof}
  Set $\stackS = \stackT = \EmptyStack{\Omega}$ in
  Lemma~\ref{lem-SearchBSGS-is-correct}.
\end{proof}

\section{Searching with a fixed sequence of left-hand
stacks}\label{sec-r-base}

In this section, we discuss a consequence of our definitions and the setup of
our algorithms, which enables a significant performance optimisation, and which
allows us to use a special kind of refiner.
This idea was inspired by, and is closely related to, the $\mathfrak{R}$-base
technique of Jeffrey Leon~\cite[Section~6]{leon1991} for partition backtrack
search, although we present the idea quite differently.

Recall that Algorithms~\ref{alg-basic-search} and~\ref{alg-bsgs-search} are
organised around a pair of labelled digraph stacks
(the stacks are equal in the \Call{SearchBSGS}{} procedure),
with both stacks initially equal to $\EmptyStack{\Omega}$.
We observe that when each of these algorithms is executed with a particular
input,
then in each branch of the search,
the left-hand stack is modified
by appending the same sequence of stacks to it, up to the end of branch.
(Note that different branches can have different lengths.)

This is because the stacks in Algorithms~\ref{alg-basic-search}
and~\ref{alg-bsgs-search} are only modified by appending stacks produced by refiners
and splitters,
and because decisions about the progression of the algorithm are made according
to the size of the value of the isomorphism approximator.
By the definition of a refiner as a pair of functions of one
variable (Definition~\ref{defn-refiner}), the left-hand stack returned by a
refiner depends only on the left-hand stack it is given;
by Definition~\ref{defn-splitter}\ref{item-splitter-invariant},
the left-hand stack defined by a splitter depends only on the given left-hand stack;
and by Definition~\ref{defn-approx-iso}\ref{item-approx-right-coset-of-aut},
the size of the value of an isomorphism approximator is either zero, in which
case the current branch immediately ends, or it depends only on the left-hand
stack that is given.

Therefore we can store the new stacks that are appended to the
left-hand stack as they are constructed,
and simply recall them as they are needed later.
This means that, on most occasions, when applying a refiner,
we recall the value for the left stack, and compute only the value for the
right-hand stack.
This optimisation significantly improves the performance of the \Call{Refine}{}
procedure.

\subsection{Constructing and applying refiners via the fixed sequence of
left-hand stacks}\label{sec-refiners-via-stack}

We saw in Lemmas~\ref{lem-simple-refiner} and~\ref{lem-refiner-right-coset}
that any refiner for a non-empty set
is derived from a function $f$ from $\Stacks{\Omega}$ to itself satisfying
$f(\stackS^{g}) = {f(\stackS)}^{g}$ for certain
$g \in \Sym{\Omega}$.
In this section, we give an example that demonstrates a difficulty in satisfying
this condition, and a general method for giving refiners that overcome
this difficulty.
We use such refiners for groups and cosets in our experiments.

\begin{example}\label{ex-refiner-fixed-stack}
  Let $\Omega = \{1, \ldots, 6\}$ and
  \(G = \< (1\,2), (3\,4), (5\,6), (1\,3\,5)(2\,4\,6) \>\),
  and let $\Gamma$ be the labelled digraph on $\Omega$ without arcs where
  $\labelFunc(1)=\exLabel{black}$, $\labelFunc(2)=\exLabel{grey}$, and the
  remaining vertices have label $\exLabel{white}$. Finally, define $\stackS =
  [\Gamma]$ and $\stackT = [\Gamma^{(1\,3\,5)(2\,4\,6)}]$.

  Suppose that we are
  searching for an intersection $D$ of subsets of $\Sym{\Omega}$, one
  of which is $G$, and suppose that $\Iso{\stackS}{\stackT}$
  overestimates the solution.
  We wish to give a refiner $(f,f)$ for $G$, where the function $f$
  takes into account that the elements of $D$ respect the
  orbit structure of $G$.

  Since $D \subseteq \Iso{\stackS}{\stackT}$,
  elements of $D$ induce isomorphisms from $\stackS$ to $\stackT$.
  In particular, if $\fixedFunc$ is a fixed-point approximator
  (see Definition~\ref{defn-approx-fixed}),
  then each element of $D$ maps the list $\Fixed{\stackS}$ to the
  list $\Fixed{\stackT}$.
  For illustration, assume that
  $\Fixed{\stackS} = [1,2]$ and $\Fixed{\stackT} = [3,4]$,
  and let $G_{[1,2]} = \< (3\,4) (5\,6) \>$ and $G_{[3,4]} = \< (1\,2) (5\,6) \>$
  denote the stabilisers of $[1,2]$ and $[3,4]$ in $G$, respectively.
  Therefore, since each element $x \in D$ maps $[1,2]$ to $[3,4]$
  and is contained in $G$, it follows that
  $x \in G_{[1,2]} \cdot h$,
  where $h$ is any permutation in $G$ that maps
  $[1,2]$ to $[3,4]$, for instance $h\coloneqq(1\,3\,5)(2\,4\,6) \in G$.
  This means that we can define $f(\stackS)$ and $f(\stackT)$ in
  terms of the orbits of $G_{[1,2]}$ and $G_{[3,4]}$.

  One option is to define $f(\stackS) = [\Gamma_{\mathcal{U}}]$
  as in Example~\ref{ex-set-of-sets}, for the set of orbits 
  $\mathcal{U} \coloneqq \{\{1\}, \{2\}, \{3,4\}, \{5,6\}\}$
  of $G_{[1,2]}$ on $\Omega$,
  and to define $f(\stackT) = [\Gamma_{\mathcal{V}}]$
  similarly for the set $\mathcal{V} \coloneqq \{\{1,2\},
  \{3\}, \{4\}, \{5,6\}\}$ of orbits of $G_{[3,4]}$ on $\Omega$.
  This is valid, but not ideal,
  since a permutation could map $[\Gamma_{\mathcal{U}}]$ to
  $[\Gamma_{\mathcal{V}}]$
  while mapping an orbit in
  $\mathcal{U}$ to any orbit of the same size in $\mathcal{V}$.
  However, every element of $G_{[1,2]} \cdot h$ 
  maps the orbit $O \in \mathcal{U}$ to $O^{h}$.
  Therefore, this refiner
  does not eliminate some elements that,
  to us, are obviously not in $D$.

  This is unsatisfactory, and so we would like to define $f(\stackS) =
  f_{\mathcal{U}}(\stackS)$ as in Example~\ref{ex-perfect-list-of-sets} for
  some \emph{ordered} list $\mathcal{U}$ of the orbits of $G_{[1,2]}$.
  But then how should we order the orbits
  of $G_{[3,4]}$ in the corresponding way, to obtain a stack $f(\stackT)$
  such that $D \subseteq  \Iso{f(\stackS)}{f(\stackT)}$?
\end{example}

To address the problem discussed in Example~\ref{ex-refiner-fixed-stack},
we use a technique similar to that of Leon~\cite{leon1991}.
In essence, we specify a refiner iteratively during search,
when applying it to a new version of the left-hand stack.
We can make choices as we do so
(about the ordering of orbits, for example).
Then, for all right-hand stacks that we encounter,
we consult the choice made for the current left-hand stack,
and remain consistent with that.

In more detail, we create such a refiner $(f, f)$ for a subgroup $G$ of
$\Sym{\Omega}$ as follows.
Let $\fixedFunc$ be a fixed point approximator, and initially
let $\stackV_{i}$ and $F_{i}$ be empty lists for all $i \in \N_{0}$.
We describe how to apply $(f, f)$
to stacks $(\stackS, \stackT)$ of length $i \coloneqq |\stackS| = |\stackT|$.

If $\stackV_{i}$ is still empty,
then we redefine $F_{i}$ to be $\Fixed{\stackS}$ and $\stackV_{i}$
to be a non-empty labelled digraph stack on $\Omega$ whose automorphism group
contains $G_{F_{i}}$,
the stabiliser of $F_{i}$ in $G$.
For example, $\stackV_{i}$ could be a list of orbital graphs of
$G_{F_{i}}$ on $\Omega$, represented as labelled digraphs,
or it could be the length-one stack $[\Gamma_{\mathcal{U}}]$ from
Example~\ref{ex-perfect-list-of-sets}, for some arbitrarily-ordered list
$\mathcal{U}$ of the orbits of $G_{F_{i}}$ on $\Omega$.

If $\stackV_{i}$ is no longer empty, then we have already applied the refiner
to stacks of length $i$.
Since a search has at most one left-hand stack of any particular length,
this means that we have already seen the left-hand stack
$\stackS$, and defined $F_{i}$ and $\stackV_{i}$ in terms of it.

The refiner gives $f(\stackS) = V_{i}$
and either $f(\stackT) = V_{i}^{a}$ (if ${\Fixed{\stackS}}^{a} = \Fixed{\stackT}$, for some $a \in G$) or $f(\stackT) = \EmptyStack{\Omega}$.
Note that, by Definition~\ref{defn-approx-fixed}\ref{item-fixed-invariant}, if no such element $a$ exists,
then $\stackS$ and  $\stackT$ are not isomorphic via $G$,
and so there are no solutions in the current branch.
Therefore the algorithm should backtrack.

The mathematical foundation of this kind of refiner is given in
Lemma~\ref{lem-fixed-stack}.
The notation in this lemma corresponds to the notation of the preceding
paragraphs.

We may use this lemma with Lemma~\ref{lem-refiner-right-coset}
to give refiners for cosets of subgroups.

\begin{lemma}\label{lem-fixed-stack}
  Let $G \leq \Sym{\Omega}$ and let $\fixedFunc$ be a fixed-point approximator.
  For all $i \in \N_{0}$, let $\stackV_{i} \in \Stacks{\Omega}$ be a labelled
  digraph stack on $\Omega$, and let $F_{i}$ be a list of points in $\Omega$
  whose stabiliser in $G$
  is a subgroup of $\Auto{\stackV_{i}}$.

  We define a function $f$ from $\Stacks{\Omega}$ to itself as follows.  For
  each $\stackS \in \Stacks{\Omega}$, let
  \[
    f(\stackS) =
    \begin{cases}
      {(\stackV_{|\stackS|})}^{a} &
        \text{if}\
        {F_{|\stackS|}}^{a} = \Fixed{\stackS}
        \ \text{for some}\
        a \in G,\\
      \EmptyStack{\Omega} &
        \text{if no such element}\ a \in G\ \text{exists}.
    \end{cases}
  \]
  Then $(f, f)$ is a refiner for $G$.
\end{lemma}

\begin{proof}
  Note that $f$ is well-defined: if $\stackS \in \Stacks{\Omega}$ and
  $a, b \in G$ both map $F_{|\stackS|}$ to $\Fixed{\stackS}$, then
  $a b^{-1}$ stabilises $F_{|\stackS|}$, and so $a b^{-1} \in
  \Auto{\stackV_{|\stackS|}}$ by assumption; therefore
  ${(\stackV_{|\stackS|})}^{a} =
  {(\stackV_{|\stackS|})}^{b}$.

  Let $\stackS \in \Stacks{\Omega}$ and $g \in G$.
  By Lemma~\ref{lem-simple-refiner},
  it suffices to show that $f(\stackS^{g}) = {f(\stackS)}^{g}$.
  There exists an element $a \in G$ mapping $F_{|S|}$ to
  $\Fixed{\stackS}$
  if and only if there exists
  an element of $b \in G$ mapping $F_{|\stackS^{g}|} = F_{|\stackS|}$ to
  $\Fixed{\stackS}$, since $g$ maps
  $\Fixed{\stackS}$ to $\Fixed{\stackS^{g}}$ by
  Definition~\ref{defn-approx-fixed}\ref{item-fixed-invariant}.
  In that case that no such $a$ and $b$ exist, then
  $f(\stackS^{g}) = \EmptyStack{\Omega} = {f(\stackS)}^{g}$.
  Otherwise
  $f(\stackS) = {(\stackV_{|\stackS|})}^{a}$
  and $f(\stackS^{g}) = {(\stackV_{|\stackS|})}^{b}$.
  In this case, $a g b^{-1}$ fixed $F_{|\stackS|}$ pointwise,
  and so $a g b^{-1} \in \Auto{\stackV_{|\stackS|}}$ by assumption.
  Therefore
  \[
    {f(\stackS)}^{g}
    =
    {(\stackV_{|\stackS|})}^{a g}
    =
    {(\stackV_{|\stackS|})}^{b}
    =
    {(\stackV_{|\stackS^{g}|})}^{b}
    =
    {f(\stackS^{g})}.
    \qedhere
  \]
\end{proof}

\section{Experiments}\label{sec-experiments}

In this section, we provide experimental data comparing the behaviour of our
algorithms against partition backtrack, in order to
highlight the potential of our techniques.
We repeat the experiments of~\cite[Section~6]{newrefiners}, which demonstrated improvements
from using orbital graphs, and we also investigate some additional challenging problems.
In many cases we observe a significant advancement with our new techniques.
We decided not to investigate classes of problems where partition
backtrack already performs very well, or problems where we would expect all
techniques (including ours) to perform badly. Instead we have chosen problems
that are interesting and important in their own right, including ones
that we expect to be hard for many search techniques.

At the time of writing, we have focused on the mathematical theory of our
algorithms, but not on the speed of our implementations.
Therefore, 
we only analyse the size of the search required by an algorithm to solve a problem, and not the time required.
We define a \emph{search node} of a search to be an instance of the main
searching procedure being called recursively during its execution;
the \emph{size of a search} is then its number of search nodes.
If an algorithm requires zero search nodes to solve a problem, then
it solves the problem without entering recursion,
which in our situation implies that the problem has either no solutions, or
exactly one.

The size of a search should depend only on the mathematical foundation of the
algorithm, rather than on the proficiency of the programmer who implements it,
and so it allows a fair basis for comparisons.
That said, we expect that where our algorithms require significantly smaller
searches,
then with an optimised implementation,
the increased time spent at each node
will be out-weighed by the smaller number of nodes in total,
giving faster searches than partition backtrack.
This is because, in general, a backtrack search algorithm spends time at each
search node to prune the search tree and organise the search.
The computations at each node of our algorithms are largely
digraph-based, and the very high performance of digraph-based computer programs
such as \bliss~\cite{bliss} and \nauty~\cite{practical2} suggests that, in
practice, these kinds of computations should be cheap.

For the problems that we investigate in
Sections~\ref{sec-grid-groups}--\ref{sec-intransitive-intersection}, we compare
the following techniques:
\begin{enumerate}[label=\textrm{(\roman*)}]
  \item\textsc{Leon}:
    Standard partition backtrack search,
    as described by Jeffrey Leon~\cite{leon1997, leon1991}.

  \item\textsc{Orbital}:
    Partition backtrack search with orbital graph refiners,
    as in~\cite{newrefiners}.

  \item\textsc{Strong}:
    Backtrack search with labelled digraphs,
    using the
    isomorphism and fixed-point approximators from
    Definition~\ref{defn-strong-approx}
    and the splitter from Definition~\ref{defn-point-splitter}.

  \item\textsc{Full}:
    Backtrack search with labelled digraphs,
    using the isomorphism and fixed-point approximators
    from Definition~\ref{defn-nauty-approx}
    and the splitter from Definition~\ref{defn-point-splitter}.
\end{enumerate}

The \textsc{Leon} technique is roughly the same as backtrack search with
labelled digraphs, where the digraphs are not allowed to
have arcs.
The \textsc{Orbital} technique is roughly the same as backtrack search with
labelled digraphs using the `weak equitable approximation' isomorphism and
fixed-point approximators from Definition~\ref{defn-weak-approx}.

The
\textsc{Strong} technique considers all labelled digraphs in the stack simultaneously to make its approximations, while the
\textsc{Full} technique computes isomorphisms and fixed points exactly,
rather than approximating them, and so in principle it is the most expensive of
the four methods.

We require refiners for groups given by generators, for cosets of such groups, for set stabilisers, and for unordered partition stabilisers.
We describe the refiners for set and unordered partition stabilisers in
Section~\ref{sec-grid-groups}.
For \textsc{Leon} we use the group and coset refiner described
in~\cite{leon1997}.
For \textsc{Orbital} we use the \textbf{DeepOrbital} group and coset refiner from~\cite{newrefiners},
although we get similar results for all the refiners described in that paper.
For the \textsc{Strong} and \textsc{Full} techniques,
we use refiners of the kind described in Section~\ref{sec-refiners-via-stack}
for groups and cosets, using orbits and orbital graphs.
These algorithms are similar to \textbf{DeepOrbital}, except they return the
created digraphs instead of filtering them internally.
For \textsc{Strong} and \textsc{Full}, we use the splitter given in
Definition~\ref{defn-point-splitter}.

We performed our experiments using the \textsc{GraphBacktracking}~\cite{GAPpkg}
and \textsc{BacktrackKit}~\cite{BTpkg} packages for \textsc{GAP}~\cite{GAP4}.
\textsc{BacktrackKit} provides a simple implementation of the algorithms in
\cite{newrefiners,leon1997,leon1991}, and provides a base for
\textsc{GraphBacktracking}. We note that where we reproduce experiments from
\cite{newrefiners}, we produce the same sized searches.

\subsection{Set stabilisers and partition stabilisers in grid
groups}\label{sec-grid-groups}

We first explore the behaviour of the four techniques on stabiliser problems
in grid
groups. This setting was previously considered
in~\cite[Section~6.1]{newrefiners}, and as mentioned there, these kinds of
problems arise in numerous real-world situations.

\begin{definition}[\mbox{Grid
  group~\cite[Definition~36]{newrefiners}}]\label{def-gridgroup}
  Let \(n \in \N\) and \(\Omega = \{1, \ldots, n\}\).  The direct product
  \(\Sym{\Omega} \times \Sym{\Omega}\) acts faithfully on the Cartesian product
  \(\Omega \times \Omega\) via
  \({(\alpha, \beta)}^{(g, h)} = (\alpha^{g}, \beta^{h})\)
  for all \(\alpha, \beta \in \Omega\) and \(g, h \in \Sym{\Omega}\).
  The \emph{\(n \times n\) grid group} is the image of the embedding of
  \(\Sym{\Omega} \times \Sym{\Omega}\) into \(\Sym{\Omega \times \Omega}\)
  defined by this action.
\end{definition}

Let \(n \in \N\) and \(\Omega = \n\), and let \(G \leq \Sym{\Omega
\times \Omega}\) be the \(n \times n\) grid group.  If we consider \(\Omega
\times \Omega\) to be an \(n \times n\) grid, where the sets of the form
\(\set{(\alpha, \beta)}{\beta \in \Omega}\) and \(\set{(\beta, \alpha)}{\beta
\in \Omega}\) for each \(\alpha \in \Omega\) are the rows and columns of the
grid, respectively, then \(G\) is the subgroup of \(\Sym{\Omega \times \Omega}\)
that preserves the set of rows and the set of columns.

We repeat the experiments in~\cite{newrefiners},
and add an unordered partition stabiliser problem:

\begin{enumerate}[label=\textrm{(\roman*)}]
  \item\label{exp-set-stab-grid-1}
    Compute the stabiliser in $G$ of a subset of \(\Omega \times \Omega\) of
    size \(\lfloor n^{2} / {2} \rfloor\).

  \item\label{exp-set-stab-grid-2}
    Compute the stabiliser in $G$ of a subset of \(\Omega \times \Omega\) with
    \(\lfloor n / {2} \rfloor\) entries in each grid-row.

  \item\label{exp-part-stab-grid-1}
    If \(2 \mathop{\vert} n\), then compute the stabiliser in $G$ of an unordered
    partition of \(\Omega \times \Omega\) that has two cells, each of size
    \(n^{2} / 2\).

\end{enumerate}

As in~\cite[Section~6.1]{newrefiners}, we compute with the \(n \times n\) grid
group as a subgroup of \(\Sn{n^{2}}\), and the
algorithms have no prior knowledge of the grid structure that
the group preserves.

For \textsc{Leon} and \textsc{Orbital}, we refine for a set stabiliser as in~\cite{leon1997}.
For \textsc{Strong} and \textsc{Full}, our refiner for set stabiliser
is the one from Example~\ref{ex-perfect-list-of-sets}.
The stabiliser in $\Sn{{n}^2}$ of an unordered partition with two parts of size \({n^{2}}/2\)
is a subgroup isomorphic to the wreath product \(\Sn{n^2/2} \wr \Sn{2}\).
For each technique, for an unordered partition stabiliser,
we directly use the group refiner for this subgroup.

Tables~\ref{tab-exp-set-stab-grid}
and~\ref{tab-exp-part-stab-grid} show the results concerning the search size
required to solve 50 random problems each of
types~\ref{exp-set-stab-grid-1},~\ref{exp-set-stab-grid-2},
and~\ref{exp-part-stab-grid-1}.
An entry in the `Zero\%' column shows the percentage of problems that
an algorithm solved with a search of size zero.
These columns are omitted when they are all-zero.

\begin{table}[!ht]
  \centering
  \small
  \begin{tabular}{r r r r r r r r r r r r}
    \hline
                      &&&& \multicolumn{2}{c}{\textsc{Orbital},}
                     &&&&& \multicolumn{2}{c}{\textsc{Orbital},} \\
     && {\textsc{Leon}} && \multicolumn{2}{c}{\textsc{Strong}, \textsc{Full}}
    &&& {\textsc{Leon}} && \multicolumn{2}{c}{\textsc{Strong}, \textsc{Full}}
    \\
    \cline{3-3}
    \cline{5-6}
    \cline{9-9}
    \cline{11-12}
    \(n\) &&
    {\scriptsize Median} && {\scriptsize Median} & {\scriptsize Zero\%} &&&
    {\scriptsize Median} && {\scriptsize Median} & {\scriptsize Zero\%}
    \\
    \hline
     3 &&    4   && 2 &  22   &&&     7   && 2 &   0 \\
     4 &&    8   && 0 &  50   &&&     8   && 2 &   0 \\
     5 &&   16   && 2 &  44   &&&    13   && 2 &   0 \\
     6 &&   23   && 0 &  68   &&&    34   && 2 &  20 \\
     7 &&   34   && 0 &  74   &&&    41   && 0 &  54 \\
     8 &&   46   && 0 &  90   &&&    92   && 0 &  68 \\
     9 &&   58   && 0 &  92   &&&   108   && 0 &  54 \\
    10 &&   75   && 0 &  88   &&&   290   && 0 &  86 \\
    11 &&  107   && 0 &  94   &&&   262   && 0 &  90 \\
    12 &&  124   && 0 & 100   &&&  1085   && 0 &  92 \\
    13 &&  155   && 0 & 100   &&&   788   && 0 &  98 \\
    14 &&  185   && 0 &  96   &&& 21774   && 0 &  96 \\
    15 &&  216   && 0 &  98   &&&  2471   && 0 & 100 \\
    \cline{1-1}
    \cline{3-6}
    \cline{9-12}
     && \multicolumn{4}{c}{Problem~\ref{exp-set-stab-grid-1}}
    &&& \multicolumn{4}{c}{Problem~\ref{exp-set-stab-grid-2}}
    \\
    \hline
  \end{tabular}
  \caption{
    Search sizes for 50 instances of
    Problems~\ref{exp-set-stab-grid-1} and~\ref{exp-set-stab-grid-2}
    in the $n \times n$ grid group.
  }\label{tab-exp-set-stab-grid}
\end{table}

\begin{table}[!ht]
  \centering
  \small
  \begin{tabular}{r r r r r r r r}
    \hline
    && \multicolumn{1}{c}{\textsc{Leon}}
    && \multicolumn{1}{c}{\textsc{Orbital}}
    && \multicolumn{2}{c}{\textsc{Strong}, \textsc{Full}}
    \\
    \cline{3-3}
    \cline{5-5}
    \cline{7-8}
    \(n\)
    && {\scriptsize Median} && {\scriptsize Median}
    && {\scriptsize Median} & {\scriptsize Zero\%}
    \\
    \hline
    4   &&  16   &&  16   && 5 &  24 \\
    6   &&  44   &&  36   && 0 &  66 \\
    8   &&  82   &&  64   && 0 &  82 \\
    10  && 129   && 100   && 0 &  88 \\
    12  && 206   && 144   && 0 &  96 \\
    14  && 317   && 196   && 0 & 100 \\
    16  && 504   && 256   && 0 & 100 \\
    18  && 664   && 324   && 0 &  98 \\
    \hline
  \end{tabular}
  \caption{
    Search sizes for 50 instances of
    Problem~\ref{exp-part-stab-grid-1}
    in the $n \times n$ grid group.
  }\label{tab-exp-part-stab-grid}
\end{table}

In~\cite[Section~6.1]{newrefiners}, the \textsc{Orbital} algorithm was much
faster than the classical \textsc{Leon} algorithm at solving
problems of types~\ref{exp-set-stab-grid-1} and~\ref{exp-set-stab-grid-2}.  In
Table~\ref{tab-exp-set-stab-grid}, we see why: \textsc{Orbital} typically
requires no search for these problems. \textsc{Leon} used a total of 65,834
nodes to solve all problems in Problem~\ref{exp-set-stab-grid-1}, and 37,882,616
nodes for Problem~\ref{exp-set-stab-grid-2}, while \textsc{Orbital} required 567
for Problem~\ref{exp-set-stab-grid-1} and 1073 for
Problem~\ref{exp-set-stab-grid-2}.  The same numbers of nodes were also required
for both \textsc{Strong} and \textsc{Full}, since there is no possible
improvement.

In Table~\ref{tab-exp-part-stab-grid} for Problem~\ref{exp-part-stab-grid-1}, however, we clearly see the
benefits of our new techniques. Partition backtrack~--~\textsc{Leon} and \textsc{Orbital}~--~takes an
increasing number of search nodes, with 140,177 nodes required for \textsc{Leon}
and 57,120 nodes for \textsc{Orbital} to solve all instances.  But \textsc{Strong} 
is powerful enough in almost all cases to solve these same problems
without search, requiring only 450 nodes to solve all problem instances.

\subsection{Intersections of primitive groups with symmetric wreath
products}\label{sec-primitive-intersection}

Next, as in~\cite[Section~6.2]{newrefiners}, we consider intersections of
primitive groups with wreath products of symmetric groups.  
To construct these problems, we
use the primitive groups library, which is included in the
\textsc{PrimGrp}~\cite{primgrp} package for \textsc{GAP}.

For a given composite \(n \in \{6,\ldots,80\}\), we create the
following problems: for each primitive subgroup \(G \leq \Sn{n}\) that is
neither \(\Sn{n}\) nor the natural alternating subgroup of \(\Sn{n}\), and for
each proper divisor $d$ of $n$, we construct the wreath product
$\Sn{n/d} \wr \Sn{d}$ as a subgroup of
\(\Sn{n}\), which we then conjugate by a randomly chosen element of
\(\Sn{n}\).  Finally, we use each algorithm in turn to compute the
intersection of \(G\) with the conjugated wreath product.
We create 50 such intersection problems for each $n$, $G$, and $d$.

For each \(k \in \{6,\ldots,80\}\), we record the cumulative number of search
nodes that each technique needs to solve all of the intersection problems for
all composite \(n \in \{6,\ldots,k\}\).
We show these cumulative totals in Figures~\ref{fig-prim1trans}
and~\ref{fig-prim2trans},
separating the groups that are 2-transitive from those that are primitive but not 2-transitive, as in~\cite[Section~6.2]{newrefiners}.
Note that the number of problems increases with the numbers of divisors of \(n\)
and primitive groups of degree \(n\);
this explains the step-like structure in these figures.

\begin{figure}[!ht]
  \centering
  \includegraphics[width=0.65\textwidth,keepaspectratio,clip=true]{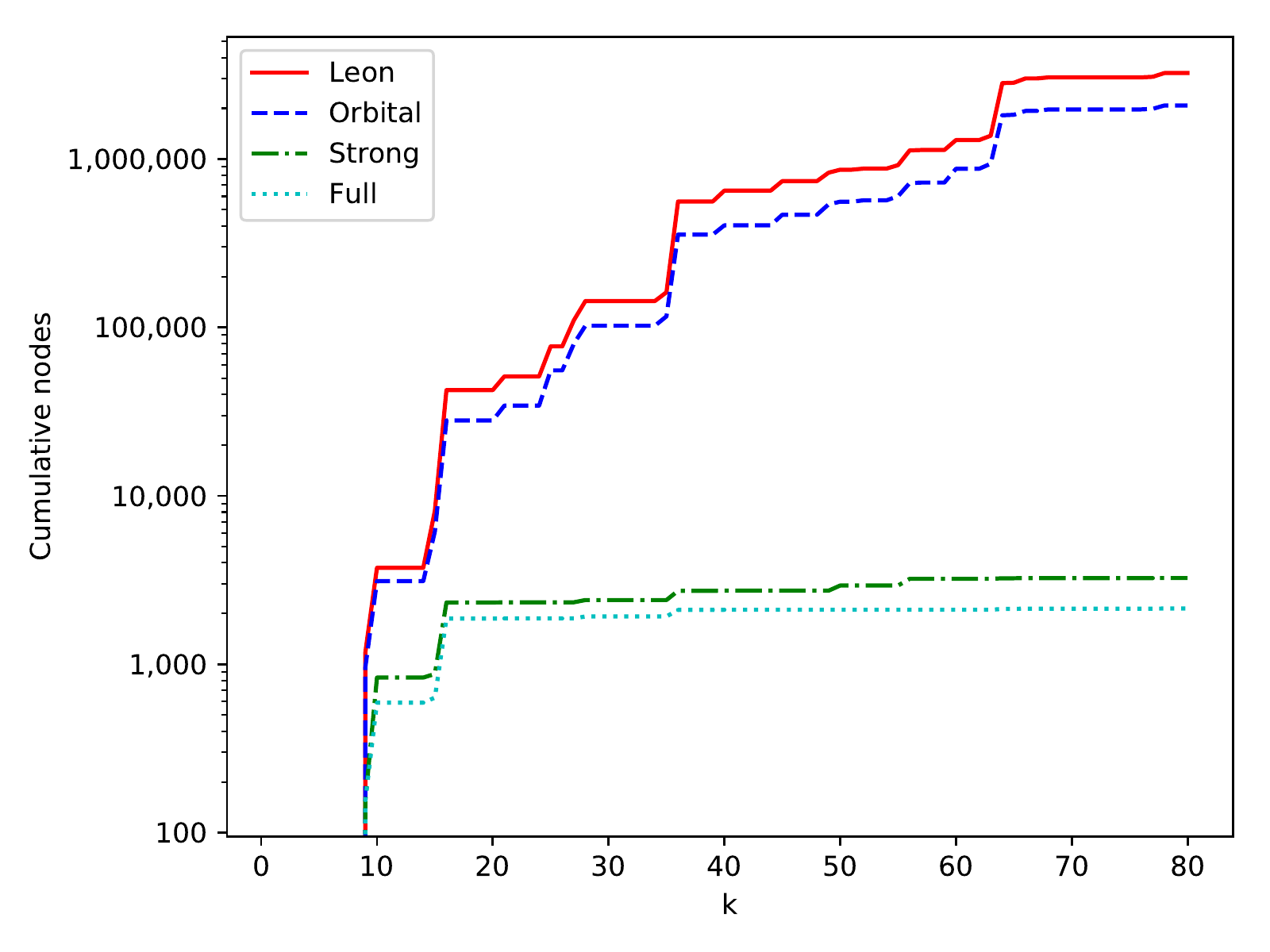}
  \caption{
    Cumulative nodes required to intersect primitive (but
    not 2-transitive) groups with wreath products of symmetric groups, for all
    problems with \(n \in \{6, \ldots, k\}\).
  }\label{fig-prim1trans}
\end{figure}

For the primitive but not 2-transitive groups, the total number of search nodes
required by the \textsc{Leon} algorithm is
3,239,403.
The \textsc{Orbital} algorithm reduces this total search size to
2,079,356,
and the cumulative search sizes for \textsc{Strong} (with 3,248 nodes) and \textsc{Full} (with 2,140 nodes) are even smaller.

This huge reduction happens because the \textsc{Strong} and
\textsc{Full} algorithms solve almost every problem without search.  Out of
40,150 experiments, the \textsc{Strong} algorithm required search for only 703,
and the \textsc{Full} algorithm required search for just 654. On the other hand,
the \textsc{Leon} and \textsc{Orbital} algorithms required search for every
problem.

\begin{figure}[!ht]
  \centering
  \includegraphics[width=0.65\textwidth,keepaspectratio,clip=true]{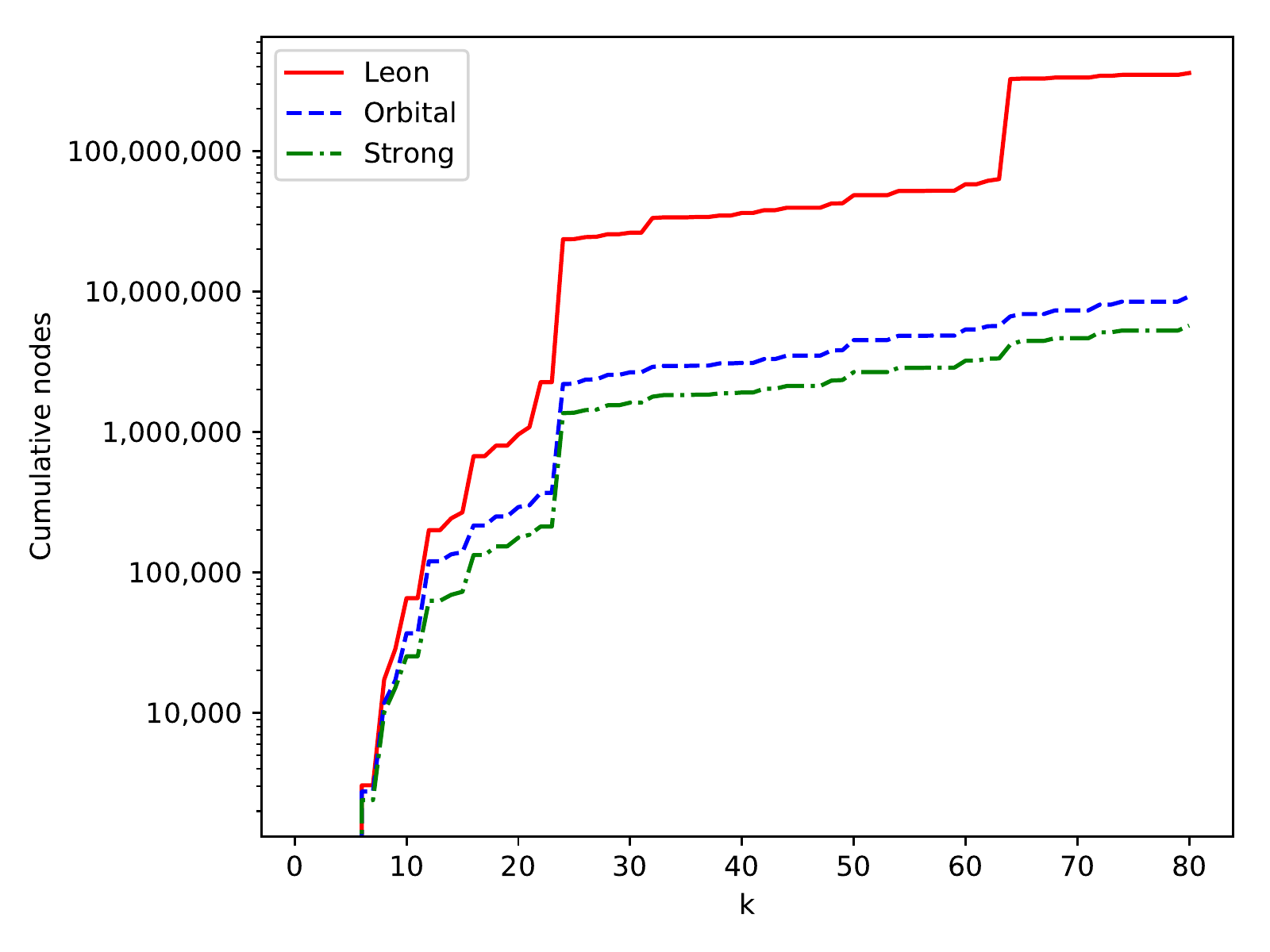}
  \caption{
    Cumulative nodes required to intersect 2-transitive groups
    with wreath products of symmetric groups, for all problems
    with \(n \in \{6, \ldots, k\}\).
    The line for \textsc{Full} is omitted, since at this scale, it is
    indistinguishable from the line for \textsc{Strong}.
  }\label{fig-prim2trans}
\end{figure}

For the intersection problems involving groups that
are at least 2-transitive, the improvement of the new techniques over the
partition backtrack algorithms is much smaller, and all of the algorithms
required a non-zero search size to solve every problem.  This was to be expected: a 2-transitive 
group has a unique orbital graph, which is a complete digraph.

\subsection{Intersections of cosets of intransitive
groups}\label{sec-intransitive-intersection}

In this section, we go beyond the experiments of~\cite[Section~6]{newrefiners}, with a problem 
that we expect to be difficult for all search techniques:
intersecting
cosets of intransitive groups that have identical orbits, and where all orbits
have the same size.

More precisely, we intersect right cosets of subdirect products of
transitive groups of equal degree.  
Given \(k, n \in \N\),
we randomly choose \(k\) transitive subgroups of \(\Sn{n}\) from the transitive
groups library \textsc{TransGrp}~\cite{transgrp}, each of which we conjugate by a
random element of \(\Sn{n}\), and we create their direct product, $G$, which we
regard as a subgroup of \(\Sn{k n}\).
Then, we randomly sample elements of \(G\) until the subgroup that they generate
is a subdirect product of \(G\).  If this subdirect product is equal to \(G\),
then we abandon the process and start again. Otherwise, the result is a
generating set for what we call a \emph{proper $(k,n)$-subdirect product}.

In our experiments, for various \(k,n \in \N\), we explore the search space
required to determine whether or not the intersections of pairs of right cosets of
different \((k,n)\)-subdirect products are empty.
To make the problems as hard as possible, we choose coset
representatives that preserve the orbit structure of the
\((k,n)\)-subdirect product.

We performed 50 random instances for each pair $(k,n)$, for all
$k, n \in \{2,\ldots,10\}$, and
we show a representative sample of this data in
Tables~\ref{tab-exp-subdirect-n} and~\ref{tab-exp-subdirect}
and Figure~\ref{fig-77nodes}.
Table~\ref{tab-exp-subdirect-n} shows results for each $n$ for all $k$ combined,
and Table~\ref{tab-exp-subdirect} gives a more in-depth view for two values of $k$.
The tables omit data for the \textsc{Full} algorithm, because it was mostly
identical to the data for the \textsc{Strong} algorithm.

\begin{table}[!ht]
  \small
  \centering
  \begin{tabular}{r r r r r r r r r r r r r r}
    &&\multicolumn{2}{c}{\textsc{Leon}}
     &&\multicolumn{3}{c}{\textsc{Orbital}}
     &&\multicolumn{3}{c}{\textsc{Strong}}\\
    \cline{3-4} \cline{6-8} \cline{10-12}
    \(n\)
    && {\scriptsize Mean} & {\scriptsize Median}
    && {\scriptsize Mean} & {\scriptsize Median} & {\scriptsize Zero\%}
    && {\scriptsize Mean} & {\scriptsize Median} & {\scriptsize Zero\%} \\
    \hline
    2   &&     3  &  2  &&     2& 2  &14  &&    2 & 2 & 14 \\
    3   &&  1418  &  7  &&    19& 0  &58  &&   19 & 0 & 59 \\
    4   &&  1250  & 12  &&    71& 0  &69  &&   62 & 0 & 70 \\
    5   &&  37924 & 30  && 15576& 10 &14  && 8803 & 0 & 54 \\
    6   &&   584  & 12  &&   254& 6  &36  &&  139 & 0 & 86 \\
    7   && 53612  & 28  && 43555& 14 &0   && 8982 & 0 & 70 \\
    8   &&  1142  &  8  &&   997& 8  &15  &&    4 & 0 & 98 \\
    9   &&  6547  &  9  &&  5562& 9  &2   &&    7 & 0 & 95 \\
    10  &&  8350  & 10  &&  6959& 10 &1   &&    7 & 0 & 97 \\
    \hline
\end{tabular}
  \caption{
    Search sizes for $(k,n)$-subdirect product
    coset intersection problems, where for each $n$,
    we ran 50 experiments for each $k \in \{2,\dots,10\}$.
  }\label{tab-exp-subdirect-n}
\end{table}

\begin{table}[!ht]
  \small
  \centering
  \begin{tabular}{r r r r r r r r r r r r r r}
    &&&\multicolumn{2}{c}{\textsc{Leon}}
     &&\multicolumn{3}{c}{\textsc{Orbital}}
     &&\multicolumn{3}{c}{\textsc{Strong}}\\
    \cline{4-5} \cline{7-9} \cline{11-13}
    \(k\) & \(n\)
    && {\scriptsize Mean} & {\scriptsize Median}
    && {\scriptsize Mean} & {\scriptsize Median} & {\scriptsize Zero\%}
    && {\scriptsize Mean} & {\scriptsize Median} & {\scriptsize Zero\%} \\
    \hline
    4  & 5  && 13683  & 30  && 6356  & 11 & 8  && 6176  & 5 & 40   \\
    4  & 6  && 376    & 18  && 335   & 6  & 8  && 87    & 0 & 76   \\
    4  & 7  && 8612   & 49  && 7065  & 43 & 0  && 6494  & 0 & 54   \\
    4  & 8  && 1133   & 8   && 365   & 8  & 14 && 0     & 0 & 100  \\
    4  & 9  && 1947   & 9   && 621   & 9  & 0  && 0     & 0 & 96   \\
    4  & 10 && 458    & 10  && 410   & 10 & 2  && 0     & 0 & 98   \\
    \hline
    8  & 5  && 119561 & 130 && 42885 & 30 & 17 && 36888 & 0 & 58   \\
    8  & 6  && 70     & 12  && 25    & 0  & 56 && 67    & 0 & 98   \\
    8  & 7  && 19731  & 49  && 11154 & 43 & 0  && 167   & 0 & 86   \\
    8  & 8  && 209    & 8   && 58    & 8  & 12 && 0     & 0 & 100  \\
    8  & 9  && 152    & 9   && 144   & 9  & 2  && 0     & 0 & 100  \\
    8  & 10 && 138    & 10  && 64    & 10 & 2  && 0     & 0 & 100  \\
    \hline
  \end{tabular}
  \caption{
    Search sizes for 50 $(k,n)$-subdirect product
    coset intersection problems.
  }\label{tab-exp-subdirect}
\end{table}

The \textsc{Strong} algorithm solved a large proportion
of problems with zero search.
As $n$ and $k$ increase, we find that
\textsc{Strong} is also able to solve almost all problems without search,
and the remaining problems with very little search. The only problems
where \textsc{Strong} does not perform significantly better are those
involving orbits of size 2 ($n = 2$).
This is not surprising as there are very few possible
orbital graphs for such groups.
We note that the problems with $n = 5$ and $7$ seem particularly difficult.
This is because transitive groups of prime degree are primitive,
and sometimes even 2-transitive,
in which case they do not have useful orbital graphs.

On the other hand, \textsc{Orbital} solved a lot fewer problems without
search, and \textsc{Leon} solved none in this way.
Although the relatively low medians show that all of the algorithms performed
quite small searches for many of the problems, we see a much starker
difference in the mean search sizes.  These means are typically dominated
by a few problems; see Figure~\ref{fig-77nodes}.

\begin{figure}[!ht]
  \centering
  \includegraphics[width=0.65\textwidth,keepaspectratio,clip=true]{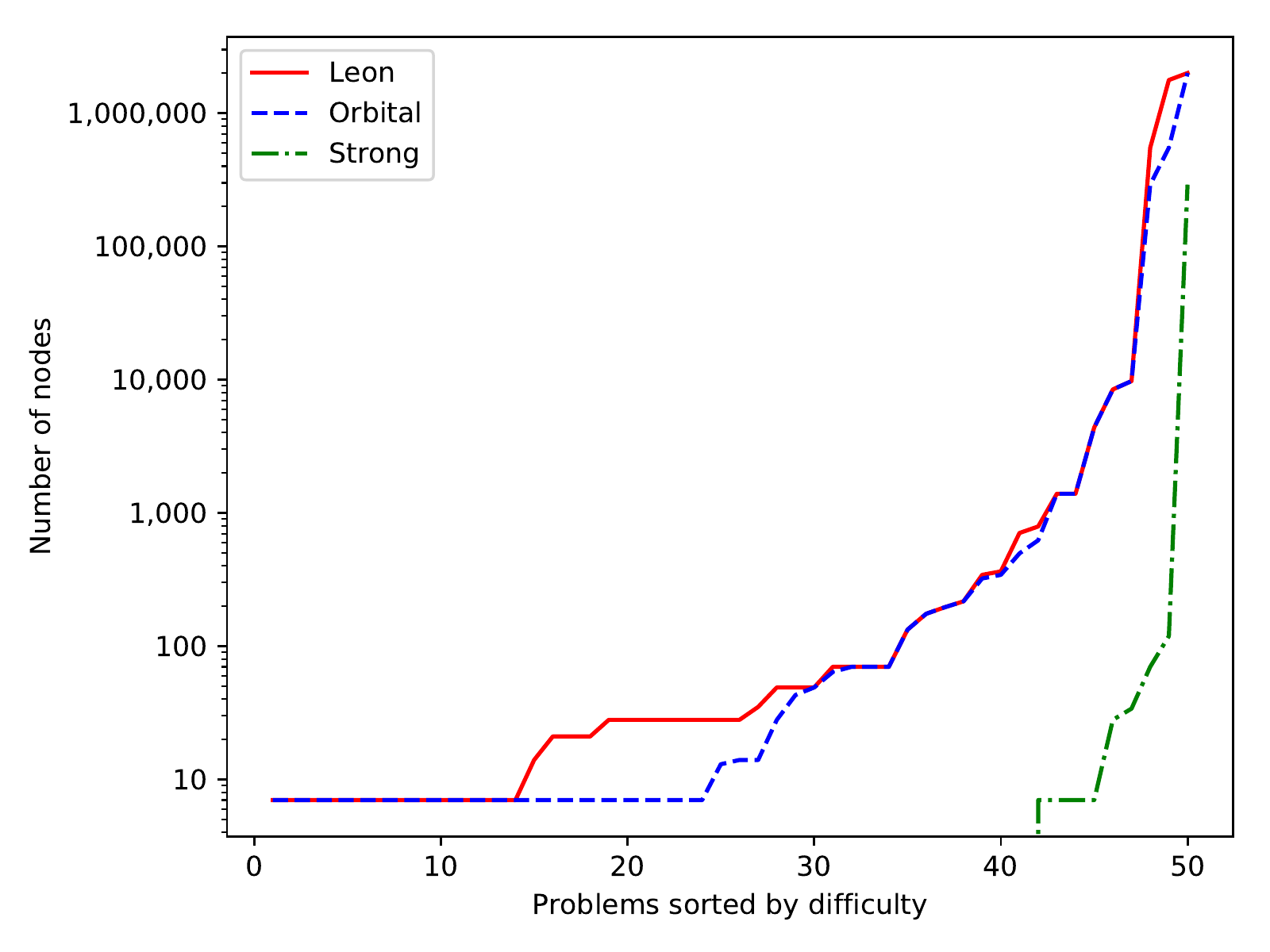}
  \caption{
    Search sizes for 50 (7,7)-subdirect product coset intersection
    problem instances.
    The data for \textsc{Full} is almost identical to the data for
    \textsc{Strong}, and is omitted.
  }\label{fig-77nodes}
\end{figure}

To give a more complete picture of how the algorithms perform,
Figure~\ref{fig-77nodes} shows the search sizes for all 50 intersections problem
that we considered for $n = k = 7$, sorted by difficulty. The data that we
collected in this case was fairly typical.
Figure~\ref{fig-77nodes} shows that
\textsc{Strong} solves almost all problems with very little or no search, and it
only requires more than 50 search nodes for the three hardest problems. On the
other hand, \textsc{Leon} and \textsc{Orbital} need more than 50 nodes for the
18 hardest problems.
All algorithms found around 30\% of the (randomly generated) problems easy to solve.

\section{Conclusions and directions for further work}\label{sec-end}

We have discussed a new search technique for
a large range of group and coset problems in
$\Sym{\Omega}$, building on the partition backtrack framework of
Leon~\cite{leon1997,leon1991}, but using stacks of labelled digraphs instead of ordered partitions. 

Our new algorithms often reduce problems that previously involved searches of
hundreds of thousands of nodes into problems that require no search, and can
instead be solved by applying strong equitable approximation to a pair of stacks.
There already exists a significant body of work on efficiently implementing
equitable partitioning and automorphism finding on
digraphs~\cite{bliss,practical2}, which we believe can be generalised to work
incrementally with labelled digraph stacks that grow in length.
We did not yet concern ourselves
with the time complexity or speed of our algorithms in this paper,
nor did we discuss their implementation details. However, we do intend for
our algorithms to be practical, and we expect that with sufficient further
development into their implementations, our algorithms should perform
competitively against, and even beat, partition backtrack for many classes of
problems. 
This requires optimising the implementation
of the algorithms that we have presented here.
In future work, we plan to show how the
algorithms described in this paper can be implemented efficiently, and compare
the speed of various methods for hard search problems.  In particular, we aim
for a better understanding of when partition backtrack is already the best
method available, and when it is worth using our methods. Further, earlier work
which used orbital graphs~\cite{newrefiners} showed that there are often
significant practical benefits to using only some of the possible orbital graphs
in a problem, rather than all of them. We will investigate whether a similar
effect occurs in our methods.

Another direction of research is the development and analysis of new types of
refiners, along with an extension of our methods. 
For example, we have seen some refiners in our examples that perfectly capture all the 
information about the set that we search for, and it is worth investigating this more.
See~\cite[Section~5.1]{directorscut} for first steps in this direction.
We could also allow
more substantial changes to the digraphs, such as adding new vertices outside of
$\Omega$. One obvious major area not addressed in this paper is normaliser and
group conjugacy problems. These problems, as well as a concept for the quality of refiners 
are addressed in ongoing work that builds on the present paper.

While the step from ordered partitions to labelled digraphs already adds some
difficulty,
we still think that it is worth considering even more intricate structures.
Why not generalise our ideas to stacks of more general combinatorial structures
defined on a set $\Omega$?  The definitions of a splitter, of an isomorphism
approximator, and of a refiner were essentially independent of the notion of a
labelled digraph, and so they~--~and therefore the algorithms~--~could work for
more general objects around which a search method could be organised.

\end{document}